\documentclass[11pt,letterpaper,reqno]{amsart}
\usepackage[colorlinks,final,backref=page,hyperindex]{hyperref}
\usepackage{xurl}
\usepackage{latexsym, amsfonts, amsmath, amssymb, amsthm, mathrsfs}
\usepackage{color}
\usepackage{xcolor}
\usepackage{hyperref}
\hypersetup{
    colorlinks=true,
    linkcolor=blue,   
    citecolor=magenta,    
    urlcolor=violet,    
    pdfborder={0 0 0} 
}

\usepackage{amssymb}
\usepackage{amsmath}
\usepackage{amsthm}
\usepackage{bbm}
\usepackage{enumerate}
\usepackage{doi}

\usepackage{bookmark}
\usepackage{hyperref}
\hypersetup{pdfstartview={FitH}}

\newtheorem{theorem}{Theorem}[section]
\newtheorem{corollary}[theorem]{Corollary}

\newtheorem{lemma}[theorem]{Lemma}
\newtheorem{proposition}[theorem]{Proposition}
\theoremstyle{definition}
\newtheorem{definition}[theorem]{Definition}
\newtheorem{remark}[theorem]{Remark}
\newtheorem{example}[theorem]{Example}

\numberwithin{equation}{section}
\allowdisplaybreaks
\newcommand{\Ztwo}{\mathbb Z_2}
\newcommand{\abs}[1]{\lvert #1\rvert}
\newcommand{\C}{\mathbb C}
\newcommand{\pa}{\partial}

\begin{document}

\title{On transposed Poisson conformal superalgebras}

\author{Hao Fang}
\address{Hao Fang, School of Mathematics, Harbin Institute of Technology, Harbin 150001, People's Republic of China}
\email{23s012028@stu.hit.edu.cn}

\author{Lamei Yuan}
\address{Lamei Yuan (corresponding author), School of Mathematics, Harbin Institute of Technology, Harbin 150001, People's Republic of China}
\email{lmyuan@hit.edu.cn}




\begin{abstract}
We introduce and study transposed Poisson conformal superalgebras, 
the $\mathbb Z_2$-graded conformal analogues of transposed Poisson algebras,
as well as their noncommutative variants. We derive a family of identities
forced by the transposed conformal super-Leibniz rule and prove that the
tensor product over $\mathbb C[\partial]$ of two such superalgebras again
carries a natural transposed Poisson conformal superalgebra structure. 
Moreover, we display a close relationship between transposed Poisson conformal superalgebras 
and Hom-Lie conformal superalgebras, and give the compatibility conditions between 
a Poisson conformal superalgebra and a transposed Poisson conformal superalgebra.
In addition, several constructions are obtained from
modified Lie conformal brackets and from Novikov-Poisson, pre-Lie
commutative, differential Novikov-Poisson, and pre-Lie Poisson conformal superalgebras.
Finally, using the known classification of Lie conformal superalgebras of
rank (1+1), we determine all compatible transposed Poisson conformal
superalgebra structures on such superalgebras.

\end{abstract}

\maketitle



\section{Introduction}

The notion of a Lie conformal superalgebra, introduced by Kac \cite{Kac98}, encodes the singular part of the operator product expansion of chiral fields in two-dimensional quantum field theory.
Closely connected to vertex algebras and formal distribution Lie superalgebras \cite{Chen25,Li96}, Lie conformal superalgebras also provide an effective framework for studying algebraic structures governed by locality \cite{Kac97Locality,Su13}.
The structure theory and representation theory of finite Lie conformal superalgebras have been systematically developed in \cite{BKLjmp10,BKLRimp06,FKja02,FKRb04}, while the cohomology theory of conformal (super)algebras was investigated in \cite{BKV99,DKjjm13}.
Significantly, the notion of a Poisson conformal superalgebra was introduced by Kolesnikov et al.~\cite{KKPjne21}, which consists of a Lie conformal superalgebra and a commutative associative conformal superalgebra satisfying certain compatibility conditions.
This compatibility provides a natural point of departure for studying variants in which the relation between the associative conformal product and the Lie conformal bracket is modified.

The present paper is concerned with the conformal super version of transposed Poisson algebras.
Recall that a Poisson algebra consists of a commutative associative product and a Lie bracket related by the usual Leibniz rule. 
Transposed Poisson algebras, introduced by Bai et al.~\cite{Bai23}, 
arise by reversing the roles of the two bilinear operations in the Leibniz rule.
They are closely related to Novikov-Poisson algebras, pre-Lie Poisson algebras, and
$\frac{1}{2}$-derivations of Lie algebras \cite{Bai23,FKLrac21}. 
The even conformal analogue of transposed Poisson algebras was considered in \cite{YF2026}. 
It is therefore natural to ask for a $\mathbb Z_2$-graded conformal theory in which the transposed compatibility is governed
simultaneously by conformal sesquilinearity and parity.

In this paper, we introduce the notion of transposed Poisson conformal superalgebras.
Such an object is a $\mathbb Z_2$-graded $\mathbb C[\partial]$-module carrying a
commutative associative conformal superalgebra structure and a Lie
conformal superalgebra structure, subject to the transposed conformal
super-Leibniz rule: for homogeneous elements $a,b,c$,
\[
2a\circ_\lambda [b_\mu c]
=
[(a\circ_\lambda b)_{\lambda+\mu}c]
+(-1)^{|a||b|}[b_\mu(a\circ_\lambda c)].
\]
When the odd part is zero, this recovers the transposed Poisson conformal algebras of \cite{YF2026}. 
If one drops the commutativity requirement on the associative conformal product, the resulting structure will be called a noncommutative
transposed Poisson conformal superalgebra.

Our first goal is to develop the basic structure theory of these
superalgebras. We derive several identities forced by the transposed
conformal super-Leibniz rule, including cyclic identities and mixed
identities involving conformal products and brackets.
These identities are useful not only as computational tools but also as structural constraints. 
As an application, we prove that the tensor product over $\mathbb C[\partial]$
of two transposed Poisson conformal superalgebras again carries a natural
transposed Poisson conformal superalgebra structure. 
We also show that, for every even element $h$, the operator $h\circ_{(0)}$
 gives rise to a Hom-Lie conformal
superalgebra, and we give necessary and sufficient conditions for a fixed
commutative associative conformal superalgebra and a fixed Lie conformal
superalgebra to form both a Poisson conformal superalgebra and a
transposed Poisson conformal superalgebra.

The second part of the paper is devoted to constructions. Starting from a given
transposed Poisson conformal superalgebra and an even element, we construct
a new transposed Poisson conformal superalgebra structure by modifying the Lie conformal bracket through the 0-th
product with that element. 
We then show that conformal super analogues of Novikov-Poisson, pre-Lie
commutative, differential Novikov-Poisson, and pre-Lie Poisson structures
give rise to transposed Poisson conformal superalgebra structures, extending the
corresponding constructions in the purely even conformal setting
\cite{YF2026}. Furthermore, we also establish a tensor product construction for
pre-Lie Poisson conformal superalgebras.

Finally, we classify compatible transposed Poisson conformal
superalgebra structures on Lie conformal superalgebras of rank (1+1).
Using the classification of such Lie conformal superalgebras in
\cite{ZCYjmh17}, we reduce the problem to the five types
$R_1,\ldots,R_5$. For each type, we determine all even conformal products
compatible with the given Lie conformal superalgebra structure.
The classification shows that the transposed conformal super-Leibniz rule
imposes strong restrictions on compatible products: except for the
abelian degenerations and a few explicitly determined nontrivial cases,
the compatible product is necessarily trivial.

The paper is organized as follows.  In Section~\ref{sec2}, we recall the
preliminaries on conformal superalgebras, Lie conformal superalgebras,
Poisson conformal superalgebras, and related structures.
In Section~\ref{sec3}, we introduce transposed Poisson conformal superalgebras and establish their fundamental
identities, tensor product construction, relation with Hom-Lie conformal superalgebras, and
compatibility with Poisson conformal superalgebras.
Section~\ref{sec4} is devoted to constructions
from Novikov-Poisson, pre-Lie commutative, differential
Novikov-Poisson, and pre-Lie Poisson conformal superalgebras.
In Section~\ref{sec5}, we present a complete classification of compatible transposed Poisson conformal 
superalgebra structures on Lie conformal superalgebras of rank (1+1).

Throughout this paper, we denote by $\mathbb{C}$ the set of complex numbers, $\mathbb{Z}_{+}$ the set of nonnegative integers. All vector spaces, linear maps, and tensor products are considered over the complex field $\mathbb{C}$.
For any vector space $A$, the space of polynomials of $\lambda$ with coefficients in $A$ is denoted by $A[\lambda]$.

\section{Preliminaries}\label{sec2}
In this section, we recall some basic notions and results concerning Lie conformal superalgebras and Poisson conformal superalgebras.

Let $V$ be a superspace that is a $\mathbb{Z}_2$-graded linear space with a direct sum $V=V_{\bar{0}}\oplus V_{\bar{1}}$.
The elements of $V_j$, $j\in\mathbb{Z}_2=\{\bar{0},\bar{1}\}$, are said to be homogeneous and of parity $j$. 
The parity of a homogeneous element $x$ is denoted by $\abs x$. In particular, $\abs 0=\bar{0}$.
Throughout what follows, if $\abs x$ occurs in an expression, then it is assumed that $x$ is homogeneous and that the expression extends to the other elements by linearity.

A \emph{superalgebra} is $\Ztwo$-graded algebra $A=A_{\bar{0}}\oplus A_{\bar{1}}$, that is, if $a\in A_\alpha$, $b\in A_\beta$, $\alpha,\beta\in\Ztwo$, then $ab\in A_{\alpha+\beta}$.
A \emph{$\Ztwo$-graded $\C[\partial]$-module} is a $\C[\partial]$-module $A=A_{\bar {0}}\oplus A_{\bar{1} }$ such that $\partial A_i\subseteq A_i$ for $i\in\Ztwo$.
The tensor product of graded spaces is graded by $\abs{a\otimes b}=\abs a+\abs b$.
\begin{definition}[\cite{DKjjm13}]
A {\bf conformal superalgebra} $(A,\circ_\lambda)$ is a vector superspace $A$ endowed with a structure of a $\mathbb{C}[\partial]$-module compatible with parity,
and with a $\lambda$-product, i.e., a linear map $A\otimes A\to \mathbb{C}[\lambda]\otimes A$, $a\otimes b\mapsto a \circ_{\lambda}b$, satisfying the sesquilinearity relations:
   \begin{align}
        (\partial a)\circ_{\lambda}b=-\lambda a\circ_{\lambda}b,\quad a\circ_{\lambda}(\partial b)=(\partial+\lambda)a\circ_{\lambda}b.
   \end{align}
A conformal superalgebra $(A,\circ_\lambda)$ is called {\bf commutative}, if it satisfies
   \begin{align}
        a\circ_\lambda b=(-1)^{\abs a\abs b}b\circ_{-\partial-\lambda}a,\quad \forall ~a,b\in {A}.
   \end{align}
\end{definition}

\begin{definition}[\cite{FKRb04}]
An {\bf associative conformal superalgebra} $(A,\circ_\lambda)$ is a left $\Ztwo$-graded $\mathbb{C}[\partial]$-module $A$ endowed with a $\mathbb{C}$-bilinear map $A\otimes A\to A[\lambda]$, denoted by $a\otimes b\mapsto a \circ_{\lambda}b$ such that $(A,\circ_\lambda)$ is a conformal superalgebra and satisfies the associativity relations:
   \begin{align}\label{ly}
      a\circ_{\lambda}\left(b\circ_{\mu} c\right)=\left(a\circ_{\lambda} b\right)\circ_{\lambda+\mu} c,\quad \forall ~a,b,c\in {A}.
   \end{align}

\end{definition}

\begin{remark}[{\cite{DAK98}}]\label{rem1}
   The following properties always hold in an associative conformal superalgebra $(A,\circ_\lambda)$:
   \begin{align}
         a\circ_{\lambda}(b\circ_{-\partial-\mu}c) & =(a\circ_{\lambda}b)\circ_{-\partial-\mu}c, \\
         a\circ_{-\partial-\lambda}(b\circ_{\mu}c) & =(a\circ_{-\partial-\mu}b)\circ_{-\partial+\mu-\lambda}c, \\
         a\circ_{-\partial-\lambda}(b\circ_{-\partial-\mu}c) & =(a_{-\partial+\mu-\lambda}b)\circ_{-\partial-\mu}c.
   \end{align}
\end{remark}
\begin{remark}[{\cite{KKPjne21}}] \label{rem3}
   In every commutative associative conformal superalgebra $(A,\circ_\lambda)$, the following identity holds:
   \begin{align}
      a\circ_\lambda\left(b\circ_\mu c\right)=(-1)^{\abs a\abs b} b\circ_\mu\left(a\circ_\lambda c\right),\quad \forall ~a,b,c\in {A}.
   \end{align}
\end{remark}

\begin{definition}[{\cite{Kac98}}]
A {\bf Lie conformal superalgebra} $A=A_{\bar{0}}\oplus A_{\bar{1}}$ is a $\Ztwo$-graded $\mathbb{C}[\partial]$-module endowed with a $\mathbb{C}$-bilinear map $A\otimes A\to A[\lambda]$, denoted by $a\otimes b\mapsto [a_{\lambda}b]$, satisfying $[{A_i}_{\lambda}A_j]\subseteq A_{i+j}[\lambda]$, $i,j\in\Ztwo$,
and the following axioms ($a,b,c\in {A}$):
   \begin{align}
        &[(\partial a)_{\lambda}b]=-\lambda[a_{\lambda}b],~[a_{\lambda}(\partial b)]=(\partial+\lambda)[a_{\lambda}b], &&\text{(Conformal sesquilinearity)}\\
        &[a_{\lambda}b]=-(-1)^{\abs a\abs b}[b_{-\lambda-\partial}a], &&\text{(Skew-symmetry)}\\
        &[a_{\lambda}[b_{\mu}c]]=[[a_{\lambda}b]_{\lambda+\mu}c]+(-1)^{\abs a\abs b}[b_{\mu}[a_{\lambda}c]]. &&\text{(Jacobi identity)} \label{Jacobi}
   \end{align}
A Lie conformal superalgebra is called \emph{finite} if it is finitely generated as a $\mathbb{C}[\partial]$-module. The \emph{rank} of a conformal superalgebra is its rank as a $\mathbb{C}[\partial]$-module.
\end{definition}

\begin{example}[{\cite{CSKajm97}}]
The simplest examples of a conformal superalgebra is the \emph{current conformal superalgebra} associated to a finite-dimensional Lie superalgebra $\mathfrak g$:
$\operatorname{Cur}\mathfrak g=\mathbb C[\partial]\otimes\mathfrak g$, with the \(\lambda\)-bracket defined by
\begin{equation}
(f(\partial) \otimes a)_\lambda(g(\partial) \otimes b):=f(-\lambda) g(\lambda+\partial) \otimes[a, b],  ~~~ f(\partial), g(\partial) \in \mathbb{C}[\partial], ~~~a, b \in \mathfrak{g},
\end{equation}
and the \emph{Virasoro conformal algebra} $\mathrm{Vir}=\mathbb C[\partial]L$ with the \(\lambda\)-bracket: $[L_\lambda L]=(\partial+2\lambda)L$.
\end{example}

\begin{example}[{\cite{CSKajm97}}] (Neveu--Schwarz Lie conformal superalgebra)
Let $\mathrm{NS}=\mathbb{C}[\partial]L\oplus \mathbb{C}[\partial]G$ be a free \(\mathbb{Z}_2\)-graded \(\mathbb{C}[\partial]\)-module. Define
   \begin{equation*}
      [L_\lambda L]=(\partial+2\lambda)L,\quad
      [L_\lambda G]=(\partial+\frac{3}{2}\lambda)G,\quad
      [G_\lambda L]=(\frac{1}{2}\partial+\frac{3}{2}\lambda)G,\quad
      [G_\lambda G]=L,
   \end{equation*}
where $\mathrm{NS}_{\bar 0}=\mathbb{C}[\partial]L$ and $\mathrm{NS}_{\bar 1}=\mathbb{C}[\partial]G$.
Then \(\mathrm{NS}\) is a Lie conformal superalgebra of rank (1+1). We call it the \emph{Neveu--Schwarz Lie conformal superalgebra}.
\end{example}

Given two $\Ztwo$-graded left $\mathbb{C}[\partial]$-modules $M$ and $N$, a \emph{conformal linear map} from $M$ to $N$ is a $\mathbb{C}$-linear map $\varphi:M\rightarrow \mathbb{C}[\lambda]\otimes N$, denoted by $\varphi_\lambda:M\rightarrow N$, such that $$\varphi_\lambda(\partial a)=(\partial+\lambda)\varphi_\lambda(a),\quad\forall ~ a\in M.$$
\begin{definition}[{\cite{FKRb04}}]
A \emph{conformal derivation} of a Lie conformal superalgebra $A$ is a conformal endomorphism $\phi_\lambda$ of $A$ such that for any homogeneous $x,y\in A$,
\begin{equation}
\phi_\lambda [x_\mu y]=[(\phi_\lambda x)_{\lambda+\mu}y]+(-1)^{\abs x \abs \phi}[x_\mu(\phi_\lambda y)].
\end{equation}
\end{definition}

More generally, we define a \emph{conformal $\delta$-derivation} of a Lie conformal superalgebra $A$ as a conformal endomorphism $\phi_\lambda$ of $A$ such that 
\begin{equation}
\phi_\lambda [x_\mu y]=\delta\big([(\phi_\lambda x)_{\lambda+\mu}y]+(-1)^{\abs x \abs \phi}[x_\mu(\phi_\lambda y)]\big), \quad\forall ~ x,y\in A.
\end{equation}
Therefore, a conformal $1$-derivation of $(A,[\cdot_{\lambda}\cdot])$  is exactly an ordinary conformal derivation.

We will also work with ordinary derivations of conformal algebras. An \emph{ordinary derivation} of a Lie conformal superalgebra $A$ is a
$\mathbb{C}[\partial]$-linear endomorphism $d$ of $A$ such that for any homogeneous $x,y\in A$,
\begin{equation}
d([x_\mu y])=[d(x)_\mu y]+(-1)^{\abs x \abs d}[x_\mu d(y)].
\end{equation}

\begin{definition}[{\cite{YCHase17}}]
A {\bf Hom-Lie conformal superalgebra} $A=A_{\bar{0}}\oplus A_{\bar{1}}$ is a $\Ztwo$-graded $\mathbb{C}[\partial]$-module equipped with an even linear endomorphism $\alpha$ satisfying $\alpha\partial=\partial\alpha$, and a $\mathbb{C}$-bilinear map $A \otimes A \rightarrow A [\lambda]$, $a\otimes b\mapsto [a_{\lambda}b]$,
such that $[{A_i}_{\lambda}A_j]\subseteq A_{i+j}[\lambda]$, $i,j\in\Ztwo$, and the following axioms hold for all $a,b,c \in A$:
   \begin{align}
      &[(\partial a)_\lambda b]=-\lambda[a_\lambda b],~[a_\lambda (\partial b)] =(\partial+\lambda)[a_\lambda b], &&\text{(Conformal sesquilinearity)} \\
      &{[a_\lambda b] } =-(-1)^{\abs a\abs b}[b_{-\lambda-\partial} a], &&\text{(Skew-symmetry)}\\
      &{[\alpha(a)_\lambda[b_\mu c]] } =[[a_\lambda b]_{\lambda+\mu} \alpha(c)]+(-1)^{\abs a\abs b}[\alpha(b)_\mu[a_\lambda c]]. &&\text{(Hom-Jacobi identity)}
   \end{align}
\end{definition}

The notion of a left-symmetric conformal algebra was introduced by Hong and Li \cite{HL15}. Its super and Hom-super analogues appear naturally in the study of Novikov conformal superalgebras and Hom-Gel'fand-Dorfman conformal superbialgebras (cf. \cite{CThjms24,Hongjlt16}).
\begin{definition}\label{defi10}
A {\bf left-symmetric conformal superalgebra} $A=A_{\bar{0}}\oplus A_{\bar{1}}$ is a left $\Ztwo$-graded $\mathbb{C}[\partial]$-module endowed with a $\mathbb{C}$-bilinear map $A\otimes A\to A[\lambda]$, denoted by $a\otimes b\mapsto a \ast_{\lambda}b$ such that $(A,\ast_\lambda)$ is a conformal superalgebra and satisfies the following identity ($a,b,c\in {A}$):
   \begin{align}
        (a\ast_\lambda b)\ast_{\lambda+\mu}c-a\ast_\lambda(b\ast_\mu c)=(-1)^{\abs a\abs b}((b\ast_\mu a)\ast_{\lambda+\mu}c-b\ast_\mu(a\ast_\lambda c)).\label{eq:2-18}
   \end{align}
A {\bf (left) Novikov conformal superalgebra} $(A,\ast_\lambda)$ is a left-symmetric conformal superalgebra satisfying
   \begin{align}
      (a\ast_\lambda b)\ast_{\lambda+\mu}c=(-1)^{\abs b\abs c}(a\ast_\lambda c)\ast_{-\mu-\partial}b,\quad \forall ~a,b,c\in {A}.
   \end{align}
When $A_{\bar 1}=\{0\}$, $A$ is called a \emph{(left) Novikov conformal algebra}.
\end{definition}

By definition \eqref{defi10}, we can easily obtain
\begin{proposition} \label{propo1}
   Associative conformal superalgebras are left-symmetric conformal superalgebras. If $(A,\ast_\lambda)$ is a left-symmetric conformal superalgebra, then the $\lambda$-bracket
   \begin{align*}
      [a_\lambda b]=a\ast_\lambda b-(-1)^{\abs a\abs b}(b\ast_{-\lambda-\partial}a),\quad \forall~ a,b \in {A}
   \end{align*}
   defines a Lie conformal superalgebra $(A,[\cdot_{\lambda}\cdot])$.
\end{proposition}

\begin{definition}[{\cite{KKPjne21}}]
A \textbf{Poisson conformal superalgebra} (PCSA) is a $\Ztwo$-graded $\mathbb{C}[\partial]$-module $A$ endowed with two $\mathbb{C}$-bilinear maps $A \otimes A \rightarrow A[\lambda]$, denoted by $a\otimes b\mapsto a \circ_{\lambda}b$ and $a\otimes b\mapsto [a_{\lambda}b]$, respectively, such that $(A,\circ_\lambda)$ is a {commutative associative conformal superalgebra}, $(A,[{\cdot_ \lambda \cdot}])$ is a {Lie conformal superalgebra} and the following {conformal analogue of the Leibniz rule} holds:
   \begin{align}\label{rule1}
      [a_{\lambda}(b\circ_{\mu}c)]=[a_{\lambda}b]\circ_{\lambda+\mu}c+(-1)^{\abs a\abs b}b\circ_{\mu}[a_{\lambda}c],\quad\forall~ a,b,c\in A.		
   \end{align}
If the associative $\lambda$-product $\circ_{\lambda}$ in the Poisson conformal superalgebra $A$ is noncommutative, we call it a {\emph{noncommutative PCSA}}.
\end{definition}

\begin{remark}
Relation \eqref{rule1} is equivalent to
   \begin{equation}\label{rule6}
      [(a \circ_\lambda b)_\mu c] = a \circ_\lambda [b _{\mu - \lambda} c]+(-1)^{\abs a \abs b}b \circ_{\mu - \lambda} [a _\lambda c],\quad\forall~ a,b,c\in A.
   \end{equation}
\end{remark}

\begin{definition}[{\cite{FKLrac21}}]
Let $\mathcal{L}=\mathcal{L}_{0}\oplus \mathcal{L}_{1}$ be a $\mathbb{Z}_{2}$-graded vector space equipped with two nonzero bilinear super-operations $\cdot$ and $[\ ,\ ]$. The triple $(\mathcal{L},\cdot,[\ ,\ ])$ is called a \textbf{transposed Poisson superalgebra} if $(\mathcal{L},\cdot)$ is a supercommutative associative superalgebra and $(\mathcal{L},[\ ,\ ])$ is a Lie superalgebra that satisfies the following compatibility condition
\begin{equation}
    2z\cdot [x,y]=[z\cdot x,y]+(-1)^{|x||z|}[x,z\cdot y],
    \qquad x,y,z\in \mathcal{L}_{0}\cup \mathcal{L}_{1}.
\end{equation}
\end{definition}

\begin{definition}[{\cite{YF2026}}]
   A \textbf{transposed Poisson conformal algebra} (TPCA)  is a $\mathbb{C}[\partial]$-module $\mathcal{L}$ endowed with two $\mathbb{C}$-bilinear maps
   $$\circ_{\lambda},\, [{\cdot_ \lambda \cdot}]: \mathcal{L}  \otimes \mathcal{L}  \rightarrow \mathcal{L} [\lambda],$$ such that $(\mathcal{L} ,\circ_\lambda)$ is a {commutative associative conformal algebra}, $(\mathcal{L} ,[{\cdot_ \lambda \cdot}])$ is a {Lie conformal algebra} and the following \emph{transposed conformal Leibniz rule} holds:
   \begin{align}\label{rule3-tpca}
      2 (a\circ_{\lambda}[b_{\mu}c])=[(a\circ_{\lambda}b)_{\lambda+\mu}c]+[b_{\mu}(a\circ_{\lambda}c)],\quad\forall~ a,b,c\in \mathcal{L}.	
   \end{align}
   If the associative $\lambda$-product $\circ_{\lambda}$ in the transposed Poisson conformal algebra $\mathcal{L}$ is noncommutative, we call it a \textbf{noncommutative TPCA}.  
\end{definition}

\section{Transposed Poisson conformal superalgebras}\label{sec3}

In this section, we introduce the notion of a transposed Poisson conformal superalgebra, which is a conformal analogue of the transposed Poisson superalgebra. Then we present some identities and properties of transposed Poisson conformal superalgebras.

\subsection{Definition of transposed Poisson conformal superalgebras}

\begin{definition}
   A \textbf{transposed Poisson conformal superalgebra} (TPCSA)  is a $\mathbb{Z}_{2}$-graded $\mathbb{C}[\partial]$-module $\mathcal L=\mathcal L_{\bar 0}\oplus \mathcal L_{\bar 1}$ endowed with two $\mathbb{C}$-bilinear maps
   $$\circ_{\lambda},\, [{\cdot_ \lambda \cdot}]: \mathcal{L}  \otimes \mathcal{L}  \rightarrow \mathcal{L} [\lambda],$$ such that $(\mathcal{L} ,\circ_\lambda)$ is a {commutative associative conformal superalgebra}, $(\mathcal{L} ,[{\cdot_ \lambda \cdot}])$ is a {Lie conformal superalgebra} and the following \emph{transposed conformal super-Leibniz rule} holds:
   \begin{align}\label{rule3}
      2 (a\circ_{\lambda}[b_{\mu}c])=[(a\circ_{\lambda}b)_{\lambda+\mu}c]+(-1)^{\abs a \abs b}[b_{\mu}(a\circ_{\lambda}c)],\quad\forall~ a,b,c\in \mathcal{L}.	
   \end{align}
   When $\mathcal{L}_{\bar 1}=\{0\}$, this definition reduces to that of a transposed Poisson conformal algebra.
   If the associative $\lambda$-product $\circ_{\lambda}$ in the transposed Poisson conformal superalgebra $\mathcal{L}$ is noncommutative, we call it a \emph{noncommutative TPCSA}.  
\end{definition}

\begin{remark}
	The relation \eqref{rule3} can be equivalently written as 
	\begin{align}\label{rule5}	2([a_{\lambda}b]\circ_{\mu}c)=[a_{{\lambda}}(b\circ_{\mu-\lambda}c)]-(-1)^{\abs a \abs b}[b_{\mu-\lambda}(a\circ_{\lambda}c)],\quad\forall~ a,b,c\in \mathcal{L}.
	\end{align}
\end{remark}

\begin{remark} Let $(\mathcal{L},\circ_{\lambda},[{\cdot_ \lambda \cdot}])$
be a TPCSA and let $a$ be a homogeneous element from $\mathcal{L}$. 
Then the left multiplication $L(a)$ in the associative commutative superalgebra $(\mathcal{L},\circ_{\lambda})$, defined by $$L(a)_\lambda (b)=a\circ_\lambda b, \quad\forall~ b\in\mathcal{L},$$
 gives a conformal 
$\frac 12$-derivation of the Lie conformal superalgebra $(\mathcal{L},[{\cdot_ \lambda \cdot}])$.
\end{remark}

Let $(\mathcal{L},\circ_{\lambda},[{\cdot_ \lambda \cdot}])$ be a (noncommutative) TPCSA.
Define the $n$-th products on the (commutative) associative conformal superalgebra $(\mathcal{L},\circ_{\lambda})$ and the Lie conformal superalgebra $(\mathcal{L},[{\cdot_ \lambda \cdot}])$ for $n\in \mathbb{Z}_{+}$, respectively, by
   \begin{equation*}
      a\circ_\lambda b=\sum_{n\in \mathbb{Z}_{+}}\lambda^{(n)}(a_{(n)}b),\quad[a_\lambda b]=\sum_{n\in \mathbb{Z}_{+}}\lambda^{(n)}(a_{[n]}b),
   \end{equation*}
where $\lambda^{(n)}=\lambda^{n}/{n!}$ and $a,b\in \mathcal{L}$. Then, the notions of (commutative) associative conformal superalgebra and Lie conformal superalgebra can be equivalently expressed in terms of their respective $n$-th products.
Moreover, the transposed conformal super-Leibniz rule of the (noncommutative) TPCSA can be rewritten in terms of $n$-th products of associative conformal superalgebras and Lie conformal superalgebras as follows:
   \begin{equation}\label{rule4}
      2(a_{(n)}(b_{[m]}c))=\sum_{j=0}^n\binom{n}{j}(a_{(j)}b)_{[n+m-j]}c+(-1)^{\abs a \abs b}b_{{[m]}}(a_{(n)}c),\quad\forall~ a,b,c\in \mathcal{L}.
   \end{equation}

\begin{example}
  Let $(A,\circ,[\cdot,\cdot])$ be an ordinary transposed Poisson superalgebra. Then $\mathcal{L}=\mathbb{C}[\partial]\otimes A$ equipped with
  $$a \circ_{\lambda} b=a \circ b,\quad[{a_ \lambda b}]=[a,b],\quad\forall~ a,b\in A$$
  is a transposed Poisson conformal superalgebra.
\end{example}

\begin{example}
Let $(\mathcal{L}_1 ,\circ_{1\lambda},[\cdot_{\lambda}\cdot]_1,\partial_1)$ and $(\mathcal{L}_2 ,\circ_{2\lambda},[\cdot_{\lambda}\cdot]_2,\partial_2)$ be two TPCSAs,
where $\mathcal{L}_1=(\mathcal{L}_1)_{\bar 0}\oplus(\mathcal{L}_1)_{\bar 1}$, $\mathcal{L}_2=(\mathcal{L}_2)_{\bar 0}\oplus(\mathcal{L}_2)_{\bar 1}$.
Define a $\mathbb{Z}_2$-grading on $\mathcal{L}_1\oplus\mathcal{L}_2$ by
$$(\mathcal{L}_1\oplus\mathcal{L}_2)_{j}=(\mathcal{L}_1)_{j}\oplus(\mathcal{L}_2)_{j},\quad j\in \Ztwo.$$
For all $x_1,y_1 \in \mathcal{L}_1$, $x_2,y_2 \in \mathcal{L}_2,$ define
\begin{align*}
(x_1+x_2)\circ_{\lambda}(y_1+y_2)&=x_1\circ_{1\lambda}y_1+x_2\circ_{2\lambda}y_2, \\
[(x_1+x_2)_{\lambda}(y_1+y_2)]&=[{x_1}_{\lambda}y_1]_1+[{x_2}_{\lambda}y_2]_2.
\end{align*}
Then $(\mathcal{L}_1 \oplus \mathcal{L}_2, \circ_{\lambda}, [\cdot_{\lambda}\cdot], \partial=\partial_1\oplus\partial_2)$ is a TPCSA.
\end{example}

\subsection{Identities in transposed Poisson conformal superalgebras}
There is a rich class of identities for TPCSAs, as listed below.
\begin{theorem}
   Let $(\mathcal{L} ,\circ_\lambda,[{\cdot_ \lambda \cdot}])$ be a TPCSA. Then the following identities hold:
   \begin{align}
      &(-1)^{\abs x \abs z}x\circ_{\lambda}[y_{\mu}z]+(-1)^{\abs x \abs y}y\circ_{\mu}[z_{-\partial-\lambda}x]+(-1)^{\abs y \abs z}z\circ_{-\partial-\lambda-\mu}[x_{\lambda}y]=0, \label{eqh}\\
      &(-1)^{\abs x \abs z}[x_{\lambda}y]\circ_{\mu}z+(-1)^{\abs x \abs y}[y_{\mu-\lambda}z]\circ_{-\partial-\lambda}x+(-1)^{\abs y \abs z}[z_{-\partial-\lambda}x]\circ_{-\partial-\mu+\lambda}y=0, \label{eqw}\\
      &(-1)^{\abs x \abs z}[(h\circ_{\lambda}[x_{\gamma}y])_{\lambda+\mu}z]+(-1)^{\abs x \abs y}[(h\circ_{\lambda}[y_{\mu-\gamma}z])_{-\partial-\gamma}x]\nonumber\\
      &\hphantom{(-1)^{\abs x \abs z}[(h\circ_{\lambda}[x_{\gamma}y])_{\lambda+\mu}z](-1)^{\abs x \abs y}(-1)} +(-1)^{\abs y \abs z}[(h\circ_{\lambda}[z_{-\partial-\gamma}x])_{-\partial-\mu+\gamma}y]=0, \label{eqs}\\
&(-1)^{|x||z|}[(h\circ_{\lambda}x)_{\lambda+\gamma}[y_{\mu-\gamma}z]]+(-1)^{|x||y|}[(h\circ_{\lambda}y)_{\lambda+\mu-\gamma}[z_{-\partial-\gamma}x]]\nonumber\\
&\hphantom{(-1)^{|x||z|}[(h\circ_{\lambda}x)_{\lambda+\gamma}[y_{\mu-\gamma}z]](-1)^{\abs x \abs y}(-1)^{\abs x \abs y}}+(-1)^{|y||z|}[(h\circ_{\lambda}z)_{-\partial-\mu}[x_{\gamma}y]]=0,  \label{equ}\\           
      &(-1)^{\abs x \abs z}[h_{\lambda}x]\circ_{\lambda+\gamma}[y_{\mu-\gamma}z]+(-1)^{\abs x \abs y}[h_{\lambda}y]\circ_{\lambda+\mu-\gamma}[z_{-\partial-\gamma}x]\nonumber\\
      &\hphantom{(-1)^{\abs x \abs z}[h_{\lambda}x]\circ_{\lambda+\gamma}[y_{\mu-\gamma}z]+(-1)^{\abs x \abs y}[h_{\lambda}y]\circ_\lambda~~}  +(-1)^{\abs y \abs z}[h_{\lambda}z]\circ_{-\partial-\mu}[x_{\gamma}y]=0,  \label{eql}\\
      &(-1)^{\abs v \abs x}[(u\circ_{\lambda}x)_{\mu}(v\circ_{\gamma}y)]+(-1)^{\abs u(\abs x +\abs v)}[(v\circ_{\gamma}x)_{\gamma+\mu-\lambda}(u\circ_{\lambda}y)]=2\left(u\circ_{\lambda}v\right)\circ_{\lambda+\gamma}[x_{\mu-\lambda}y], \label{ss1}\\
      &(-1)^{\abs v \abs x}x\circ_{\mu-\lambda}[u_{\lambda}(v\circ_{\gamma}y)]+(-1)^{\abs u \abs v}[(v\circ_{\gamma}x)_{\gamma+\mu-\lambda}u]\circ_{\gamma+\mu}y \nonumber\\
      &\hphantom{(-1)^{\abs x \abs z}[(h\circ_{\lambda}[x_{\gamma}y])_{\lambda+\mu}z](-1)^{\abs x \abs y}(-1)-1}  -(-1)^{\abs u \abs x}\left(u\circ_{\lambda}v\right)\circ_{\lambda+\gamma}[x_{\mu-\lambda}y]=0,\label{ss2}
   \end{align}
   for all $x,y,z,h,u,v \in \mathcal{L}$.
\end{theorem}
\begin{proof}
Let $x,y,z,h,u,v \in \mathcal{L}$. In fact, Eq.~\eqref{eqw} is equivalent to Eq.~\eqref{eqh}. By Eqs.~\eqref{rule3} and \eqref{rule5}, we have
   \begin{align*}
      &(-1)^{\abs x \abs z}{2} (x\circ_{\lambda}[y_{\mu}z])=(-1)^{\abs x \abs z}[(x\circ_{\lambda}y)_{\lambda+\mu}z]+(-1)^{\abs x (\abs y+\abs z)}[y_{\mu}(x\circ_{\lambda}z)],	\\
      &(-1)^{\abs x \abs y}{2} (y\circ_{\mu}[z_{-\partial-\lambda}x])=-(-1)^{\abs x (\abs y+\abs z)}[(y\circ_{\mu}x)_{\lambda+\mu}z]-(-1)^{\abs x \abs z}[x_{\lambda}(y\circ_{\mu}z)],\\
      &(-1)^{\abs y \abs z}{2} (z\circ_{-\partial-\lambda-\mu}[x_{\lambda}y])=(-1)^{\abs x \abs z}[x_{{\lambda}}(y\circ_{\mu}z)]-(-1)^{\abs x (\abs y+\abs z)}[y_{\mu}(x\circ_{\lambda}z)].
   \end{align*}

Note that $[(x\circ_{\lambda}y)_{\lambda+\mu}z]=(-1)^{\abs x \abs y}[(y\circ_{-\partial-\lambda}x)_{\lambda+\mu}z]=(-1)^{\abs x \abs y}[(y\circ_{\mu}x)_{\lambda+\mu}z]$. Taking the sum of the three identities yields Eq.~\eqref{eqh}.

The proof of Eq.~\eqref{eqs} is more involved. First by Eq.~\eqref{rule3}, we have
   \begin{align*}
      (-1)^{\abs x \abs z}2(h\circ_{\lambda}[[x_{\gamma}y]_{\mu}z])=&(-1)^{\abs x \abs z}[(h\circ_{\lambda}[x_{\gamma}y])_{\lambda+\mu}z]\\
      &+(-1)^{\abs x \abs z+\abs h(\abs x +\abs y)}[[x_{\gamma}y]_{\mu}(h\circ_{\lambda}z)],\\
      -(-1)^{\abs x \abs z}2(h\circ_{\lambda}[x_{\gamma}[y_{\mu-\gamma}z]])=&(-1)^{\abs x\abs y }[(h\circ_{\lambda}[y_{\mu-\gamma}z])_{-\partial-\gamma}x]\\
      &+(-1)^{\abs x\abs y+\abs h(\abs y +\abs z)}[[y_{\mu-\gamma}z]_{-\partial-\lambda-\gamma}(h\circ_{\lambda}x)],\\
      (-1)^{\abs x (\abs y+\abs z)}2(h\circ_{\lambda}[y_{\mu-\gamma}[x_{\gamma}z]])=&(-1)^{\abs y\abs z}[(h\circ_{\lambda}[z_{-\partial-\gamma}x])_{-\partial-\mu+\gamma}y]\\
      &+(-1)^{\abs y\abs z+\abs h(\abs x+ \abs z)}[[z_{-\partial-\gamma}x]_{-\partial-\lambda-\mu+\gamma}(h\circ_{\lambda}y)].
   \end{align*}
Summing the three identities above and applying the Jacobi identity~\eqref{Jacobi}, we obtain
   \begin{eqnarray}\label{Id1}
      \begin{aligned}
      &(-1)^{\abs x \abs z}[(h\circ_{\lambda}[x_{\gamma}y])_{\lambda+\mu}z]+(-1)^{\abs x \abs y}[(h\circ_{\lambda}[y_{\mu-\gamma}z])_{-\partial-\gamma}x]\\
      &+(-1)^{\abs y \abs z}[(h\circ_{\lambda}[z_{-\partial-\gamma}x])_{-\partial-\mu+\gamma}y]\\
      &+(-1)^{\abs x \abs z+\abs h(\abs x +\abs y)}[[x_{\gamma}y]_{\mu}(h\circ_{\lambda}z)]
      +(-1)^{\abs x\abs y+\abs h(\abs y +\abs z)}[[y_{\mu-\gamma}z]_{-\partial-\lambda-\gamma}(h\circ_{\lambda}x)]\\
      &+(-1)^{\abs y\abs z+\abs h(\abs x+ \abs z)}[[z_{-\partial-\gamma}x]_{-\partial-\lambda-\mu+\gamma}(h\circ_{\lambda}y)]=0.
      \end{aligned}
   \end{eqnarray}
Then, applying the Jacobi identity to $[[x_{\gamma}y]_{\mu}(h\circ_{\lambda}z)]$ yields
   \begin{align*}
      [[x_{\gamma}y]_{\mu}(h\circ_{\lambda}z)]+(-1)^{\abs x \abs y}[y_{\mu-\gamma}[x_{\gamma}(h\circ_{\lambda}z)]]=[x_{\gamma}[y_{\mu-\gamma}(h\circ_{\lambda}z)]]
   \end{align*}
and applying Eq.~\eqref{rule3} to $[y_{\mu-\gamma}(h\circ_{\lambda}z)]$ yields
   \begin{align*}
      (-1)^{\abs h \abs y}[y_{\mu-\gamma}(h\circ_{\lambda}z)]=2\left(h\circ_{\lambda}[y_{\mu-\gamma}z]\right)-[(h\circ_{\lambda}y)_{\lambda+\mu-\gamma}z].
   \end{align*}
Hence, we arrive at
      \begin{align*}
      &(-1)^{\abs x \abs z+\abs h(\abs x +\abs y)}[[x_{\gamma}y]_{\mu}(h\circ_{\lambda}z)]+(-1)^{\abs x\abs y}2[(h\circ_{\lambda}[y_{\mu-\gamma}z])_{-\partial-\gamma}x]\\
      &=-(-1)^{\abs x \abs z+\abs h\abs x }[x_{\gamma}[(h\circ_{\lambda}y)_{\lambda+\mu-\gamma}z]]-(-1)^{\abs x (\abs y+\abs z)+\abs h(\abs x +\abs y)}[y_{\mu-\gamma}[x_{\gamma}(h\circ_{\lambda}z)]].
   \end{align*}
In a similar way, we have
   \begin{align*}
      &(-1)^{\abs x\abs y+\abs h(\abs y +\abs z)}[[y_{\mu-\gamma}z]_{-\partial-\lambda-\gamma}(h\circ_{\lambda}x)]+(-1)^{\abs y\abs z}2[(h\circ_{\lambda}[z_{-\partial-\gamma}x])_{-\partial-\mu+\gamma}y] \nonumber\\
      &=(-1)^{\abs x (\abs y+\abs z)+\abs h(\abs x +\abs y)}[y_{\mu-\gamma}[x_{\gamma}(h\circ_{\lambda}z)]]-(-1)^{\abs x \abs z}[[(h\circ_{\lambda}x)_{\lambda+\gamma}y]_{\lambda+\mu}z],\\
      &(-1)^{\abs y\abs z+\abs h(\abs x+ \abs z)}[[z_{-\partial-\gamma}x]_{-\partial-\lambda-\mu+\gamma}(h\circ_{\lambda}y)]+(-1)^{\abs x\abs y}2[(h\circ_{\lambda}[x_{\gamma}y])_{\lambda+\mu}z] \nonumber\\
      &=(-1)^{\abs x \abs z+\abs h\abs x }[x_{\gamma}[(h\circ_{\lambda}y)_{\lambda+\mu-\gamma}z]]+(-1)^{\abs x \abs z}[[(h\circ_{\lambda}x)_{\lambda+\gamma}y]_{\lambda+\mu}z].
   \end{align*}
Upon summing the three identities above, we obtain
   \begin{eqnarray}
      \begin{aligned}\label{Id2}
      &(-1)^{\abs x \abs z}2[(h\circ_{\lambda}[x_{\gamma}y])_{\lambda+\mu}z]+(-1)^{\abs x \abs y}2[(h\circ_{\lambda}[y_{\mu-\gamma}z])_{-\partial-\gamma}x]\\
      &+(-1)^{\abs y \abs z}2[(h\circ_{\lambda}[z_{-\partial-\gamma}x])_{-\partial-\mu+\gamma}y]\\
      &+(-1)^{\abs x \abs z+\abs h(\abs x +\abs y)}[[x_{\gamma}y]_{\mu}(h\circ_{\lambda}z)]
      +(-1)^{\abs x\abs y+\abs h(\abs y +\abs z)}[[y_{\mu-\gamma}z]_{-\partial-\lambda-\gamma}(h\circ_{\lambda}x)]\\
      &+(-1)^{\abs y\abs z+\abs h(\abs x+ \abs z)}[[z_{-\partial-\gamma}x]_{-\partial-\lambda-\mu+\gamma}(h\circ_{\lambda}y)]=0.
      \end{aligned}
   \end{eqnarray}
Taking the difference between Eq.~\eqref{Id1} and Eq.~\eqref{Id2} yields Eq.~\eqref{eqs}.
Moreover, by substituting Eq.~\eqref{eqs} into Eq.~\eqref{Id1}, we obtain
      \begin{align*}
      &(-1)^{\abs x \abs z+\abs h(\abs x +\abs y)}[[x_{\gamma}y]_{\mu}(h\circ_{\lambda}z)]
      +(-1)^{\abs x\abs y+\abs h(\abs y +\abs z)}[[y_{\mu-\gamma}z]_{-\partial-\lambda-\gamma}(h\circ_{\lambda}x)]\\
      &+(-1)^{\abs y\abs z+\abs h(\abs x+ \abs z)}[[z_{-\partial-\gamma}x]_{-\partial-\lambda-\mu+\gamma}(h\circ_{\lambda}y)]=0.
      \end{align*}
After a suitable permutation of the arguments, e.g. $$(-1)^{\abs x \abs z+\abs h(\abs x +\abs y)}[[x_{\gamma}y]_{\mu}(h\circ_{\lambda}z)]=-(-1)^{\abs y \abs z}[(h\circ_{\lambda}z)_{-\partial-\mu}[x_{\gamma}y]],$$ we get Eq.~\eqref{equ}. 

By Eq.~\eqref{rule3}, we obtain
   \begin{align}\label{yqq2}
      2\left([x_{\gamma}y]\circ_{\mu}[h_{\lambda}z]\right)=[([x_{\gamma}y]\circ_{\mu}h)_{\lambda+\mu}z]+(-1)^{\abs h (\abs x+\abs y)}[h_{\lambda}([x_{\gamma}y]\circ_{\mu}z)].
   \end{align}
Observe that
   \begin{align*}
      (-1)^{\abs h (\abs x+\abs y)}[(h\circ_{\lambda}[x_{\gamma}y])_{\lambda+\mu}z]=[([x_{\gamma}y]\circ_{-\partial-\lambda}h)_{\lambda+\mu}z]
      =[([x_{\gamma}y]\circ_{\mu}h)_{\lambda+\mu}z].
   \end{align*}
Thus, Eq.~\eqref{yqq2} can be rewritten as
   \begin{align*}
      (-1)^{\abs y \abs z}2\left([h_{\lambda}z]\circ_{-\partial-\mu}[x_{\gamma}y]\right)=(-1)^{\abs x \abs z}[(h\circ_{\lambda}[x_{\gamma}y])_{\lambda+\mu}z]+(-1)^{\abs y \abs z}[h_{\lambda}(z\circ_{-\partial-\mu}[x_{\gamma}y])].
   \end{align*}
Similarly, by Eq.~\eqref{rule5}, we have
   \begin{align*}
      &(-1)^{\abs x \abs z}2([h_{\lambda}x]\circ_{\lambda+\gamma}[y_{\mu-\gamma}z])=(-1)^{\abs x \abs y}[(h\circ_{\lambda}[y_{\mu-\gamma}z])_{-\partial-\gamma}x]+(-1)^{\abs x \abs z}[h_{{\lambda}}(x\circ_{\gamma}[y_{\mu-\gamma}z])],\\
      &(-1)^{\abs x \abs y}2([h_{\lambda}y]\circ_{\lambda+\mu-\gamma}[z_{-\partial-\gamma}x])=(-1)^{\abs y \abs z}[(h\circ_{\lambda}[z_{-\partial-\gamma}x])_{-\partial-\mu+\gamma}y]\\
      &\hphantom{(-1)^{\abs x \abs y}2([h_{\lambda}y]\circ_{\lambda+\mu-\gamma}[z_{-\partial-\gamma}x])=(-1)^{\abs x \abs y}(h\circ_{\lambda}[y_{\mu-\gamma}z])}+(-1)^{\abs x \abs y}[h_{{\lambda}}(y\circ_{\mu-\gamma}[z_{-\partial-\gamma}x])].
   \end{align*}
Summing the three identities above and applying Eqs.~\eqref{eqh} and \eqref{eqs}, we get Eq.~\eqref{eql}.

Finally, we turn to the proof of Eqs.~\eqref{ss1} and \eqref{ss2}. Since $(\mathcal{L} ,\circ_\lambda)$ is a {commutative associative conformal superalgebra}, we have
   \begin{align*}
        \left(u\circ_{\lambda}v\right)\circ_{\lambda+\gamma}[x_{\mu-\lambda}y]=u\circ_{\lambda}\left(v\circ_{\gamma}[x_{\mu-\lambda}y]\right).
   \end{align*}
By Eq.~\eqref{rule3}, we obtain
   \begin{align}\label{qq1}
      {4}(u\circ_{\lambda}\left(v\circ_{\gamma}[x_{\mu-\lambda}y]\right))=&(-1)^{\abs v \abs x}[(u\circ_{\lambda}x)_{\mu}(v\circ_{\gamma}y)]+(-1)^{\abs u(\abs x +\abs v)}[(v\circ_{\gamma}x)_{\gamma+\mu-\lambda}(u\circ_{\lambda}y)] \nonumber\\
      &+[(u\circ_{\lambda}(v\circ_{\gamma}x))_{\gamma+\mu}y]+(-1)^{\abs x(\abs u +\abs v)}[x_{\mu-\lambda}(u\circ_{\lambda}(v\circ_{\gamma}y))].
   \end{align}
Observe that
   \begin{align}\label{qq2}
      {2}\left(u\circ_{\lambda}v\right)\circ_{\lambda+\gamma}[x_{\mu-\lambda}y]&=[((u\circ_{\lambda}v)\circ_{\lambda+\gamma}x)_{\mu+\gamma}y]+(-1)^{\abs x(\abs u +\abs v)}[x_{\mu-\lambda}((u\circ_{\lambda}v)\circ_{\lambda+\gamma}y)] \nonumber\\
      &=[(u\circ_{\lambda}(v\circ_{\gamma}x))_{\gamma+\mu}y]+(-1)^{\abs x(\abs u +\abs v)}[x_{\mu-\lambda}(u\circ_{\lambda}(v\circ_{\gamma}y))].
   \end{align}
Then Eq.~\eqref{ss1} follows from taking the difference between Eq.~\eqref{qq1} and Eq.~\eqref{qq2}.

Note that the first term on the left-hand side of Eq.~\eqref{ss1} can be rewritten as
   \begin{align*}
      (-1)^{\abs v \abs x}[(u\circ_{\lambda}x)_{\mu}(v\circ_{\gamma}y)]&=(-1)^{\abs x(\abs u +\abs v)}[(x\circ_{-\partial-\lambda}u)_{\mu}(v\circ_{\gamma}y)]\\
      &=(-1)^{\abs x(\abs u +\abs v)}[(x\circ_{\mu-\lambda}u)_{\mu}(v\circ_{\gamma}y)]
   \end{align*}
and by Eq.~\eqref{rule3}, we obtain
   \begin{align}\label{f5}
      (-1)^{\abs x(\abs u +\abs v)}[(x\circ_{\mu-\lambda}u)_{\mu}(v\circ_{\gamma}y)]=(-1)&^{\abs x(\abs u +\abs v)}{2}(x\circ_{\mu-\lambda}[u_{\lambda}(v\circ_{\gamma}y)])\nonumber\\
      &-(-1)^{\abs v \abs x}[u_{\lambda}(x\circ_{\mu-\lambda}(v\circ_{\gamma}y))].
   \end{align}
For the second term, by Eq.~\eqref{rule5}, we get
   \begin{align}\label{f6}
      (-1)^{\abs u(\abs x +\abs v)}[(v\circ_{\gamma}x)_{{\gamma+\mu-\lambda}}(u\circ_{\lambda}y)]=(-1)^{\abs u(\abs x +\abs v)}&{2}([(v\circ_{\gamma}x)_{\gamma+\mu-\lambda}u]\circ_{\gamma+\mu}y)\nonumber\\
      &+[u_{\lambda}((v\circ_{\gamma}x)\circ_{\gamma+\mu-\lambda}y)].
   \end{align}
Since
   \begin{align*}
      x\circ_{\mu-\lambda}(v\circ_{\gamma}y)&=(x\circ_{\mu-\lambda}v)\circ_{\gamma+\mu-\lambda}y\\
      &=(-1)^{\abs v \abs x}(v\circ_{-\partial-\mu+\lambda}x)\circ_{\gamma+\mu-\lambda}y
      =(-1)^{\abs v \abs x}(v\circ_{\gamma}x)\circ_{\gamma+\mu-\lambda}y,
   \end{align*}
we have 
   \begin{align*}
        [u_{\lambda}(x\circ_{\mu-\lambda}(v\circ_{\gamma}y))]=(-1)^{\abs v \abs x}[u_{\lambda}((v\circ_{\gamma}x)\circ_{\gamma+\mu-\lambda}y)].
   \end{align*}
Hence, by taking the sum of Eqs.~\eqref{f5} and \eqref{f6} and applying Eq.~\eqref{ss1}, we arrive at Eq.~\eqref{ss2}.
\end{proof}

\subsection{Tensor products of transposed Poisson conformal superalgebras}
Now we define the tensor product of TPCSAs as follows:
\begin{theorem}
Let $(\mathcal{L}_1 ,\circ^{1}_{\lambda},[\cdot_{\lambda}\cdot]^1)$ and $(\mathcal{L}_2 ,\circ^{2}_{\lambda},[\cdot_{\lambda}\cdot]^2)$ be two TPCSAs.
Set $$\mathcal L:=\mathcal L_1\otimes_{\mathbb C[\partial]} \mathcal L_2,\quad \abs{a\otimes b}=\abs a+\abs b$$ and the $\mathbb C[\partial]$-module structure on $\mathcal L$ is given by
\begin{align*}
\partial(a\otimes b):=(\partial a)\otimes b = a\otimes (\partial b),
\quad a\in\mathcal L_1,\, b\in\mathcal L_2 .
\end{align*}
Define the $\lambda$-product $\circ_\lambda$ and the $\lambda$-bracket $[\cdot_\lambda \cdot]$ on $\mathcal L$ respectively by
\begin{gather*}
(a_1\otimes a_2)\circ_\lambda (b_1\otimes b_2)
= (-1)^{\abs {a_2}\abs {b_1}}(a_1\circ^{1}_\lambda b_1)\otimes (a_2\circ^{2}_\lambda b_2),\\
[(a_1\otimes a_2)_\lambda (b_1\otimes b_2)]
=  (-1)^{\abs {a_2}\abs {b_1}}\big([{a_1}_{\lambda}b_1]^{1}\otimes (a_2\circ^{2}_\lambda b_2)+(a_1\circ^{1}_\lambda b_1)\otimes [{a_2}_{\lambda}b_2]^{2}\big),
\end{gather*}
for all $a_1,b_1 \in \mathcal{L}_1$, $a_2,b_2 \in \mathcal{L}_2$.
Then $(\mathcal L,\circ_\lambda,[\cdot_\lambda\cdot])$ is a TPCSA.
\end{theorem}
\begin{proof}
Let $a_1,b_1,c_1 \in \mathcal{L}_1$, $a_2,b_2,c_2 \in \mathcal{L}_2$. It is easy to verify that $(\mathcal L,\circ_\lambda)$ is a commutative associative conformal superalgebra.
Then we show that $(\mathcal L,[\cdot_\lambda\cdot])$ is a Lie conformal superalgebra. 
\begin{itemize}
\item[(a)] {\bf{Conformal sesquilinearity}:}
\begin{align*}
&[\partial(a_1\otimes a_2)_\lambda (b_1\otimes b_2)]=[((\partial a_1)\otimes a_2)_\lambda (b_1\otimes b_2)]\\
=& (-1)^{\abs {a_2}\abs {b_1}}\big([{(\partial a_1)}_{\lambda}b_1]^{1}\otimes (a_2\circ^{2}_\lambda b_2)
+((\partial a_1)\circ^{1}_\lambda b_1)\otimes [{a_2}_{\lambda}b_2]^{2}\big)\\
=& (-1)^{\abs {a_2}\abs {b_1}}\big(-\lambda([{a_1}_{\lambda}b_1]^{1}\otimes (a_2\circ^{2}_\lambda b_2))-\lambda((a_1\circ^{1}_\lambda b_1)\otimes [{a_2}_{\lambda}b_2]^{2})\big)\\
=&-(-1)^{\abs {a_2}\abs {b_1}}\lambda \big([{a_1}_{\lambda}b_1]^{1}\otimes (a_2\circ^{2}_\lambda b_2)+(a_1\circ^{1}_\lambda b_1)\otimes [{a_2}_{\lambda}b_2]^{2}\big)\\
=&-\lambda [(a_1\otimes a_2)_\lambda (b_1\otimes b_2)].
\end{align*}
\item[(b)] {\bf{Skew-symmetry}:}
\begin{align*}
&[(b_1\otimes b_2)_{-\partial-\lambda} (a_1\otimes a_2)]\\
=&(-1)^{\abs {b_2}\abs {a_1}}\big([{b_1}_{-\partial-\lambda}a_1]^{1}\otimes (b_2\circ^{2}_{-\partial-\lambda} a_2)+(b_1\circ^{1}_{-\partial-\lambda} a_1)\otimes [{b_2}_{-\partial-\lambda}a_2]^{2}\big)\\
=&-(-1)^{\abs {a_1}\abs {b_1}+\abs {a_2}\abs {b_2}+\abs {b_2}\abs {a_1}}\big([{a_1}_{\lambda}b_1]^{1}\otimes (a_2\circ^{2}_\lambda b_2)+(a_1\circ^{1}_\lambda b_1)\otimes [{a_2}_{\lambda}b_2]^{2}\big)\\
=&-(-1)^{(\abs {a_1}+\abs {a_2})(\abs {b_1}+\abs {b_2})+\abs {a_2}\abs{b_1}}\big([{a_1}_{\lambda}b_1]^{1}\otimes (a_2\circ^{2}_\lambda b_2)+(a_1\circ^{1}_\lambda b_1)\otimes [{a_2}_{\lambda}b_2]^{2}\big)\\
=&-(-1)^{(\abs {a_1}+\abs {a_2})(\abs {b_1}+\abs {b_2})}[(a_1\otimes a_2)_\lambda (b_1\otimes b_2)].
\end{align*}
\item[(c)] {\bf{Jacobi identity}:} To verify the Jacobi identity, we set $h_0=(-1)^{|a_2||b_1|+(|a_2|+|b_2|)|c_1|}$ and compute
\begin{align*}
&[[(a_1\otimes a_2)_\lambda (b_1\otimes b_2)]_{\lambda+\mu} (c_1\otimes c_2)]-[(a_1\otimes a_2)_\lambda [(b_1\otimes b_2)_\mu (c_1\otimes c_2)]]\\
&+(-1)^{(\abs {a_1}+\abs {a_2})(\abs {b_1}+\abs {b_2})}[(b_1\otimes b_2)_\mu [(a_1\otimes a_2)_\lambda (c_1\otimes c_2)]]\\
&=(Y1)+(Y2)+(Q1)+(Q2),
\end{align*}
\end{itemize}

where
\begin{align*}
h_0(Y1)&=[{[{a_1}_{\lambda}b_1]^{1}}_{\lambda+\mu}c_1]^{1}\otimes ((a_2\circ^{2}_\lambda b_2)\circ^{2}_{\lambda+\mu} c_2)\\
&+((-1)^{\abs {a_1}\abs {b_1}}[{b_1}_{\mu}[{a_1}_{\lambda}c_1]^{1}]^{1})\otimes ((-1)^{\abs {a_2}\abs {b_2}}b_2\circ^{2}_\mu (a_2\circ^{2}_\lambda c_2))\\
&-[{a_1}_{\lambda}[{b_1}_{\mu}c_1]^{1}]^{1}\otimes (a_2\circ^{2}_\lambda (b_2\circ^{2}_\mu c_2)),\\
h_0(Y2)&=((a_1\circ^{1}_\lambda b_1)\circ^{1}_{\lambda+\mu} c_1)\otimes [{[{a_2}_{\lambda}b_2]^{2}}_{\lambda+\mu}c_2]^{2}\\
&+((-1)^{\abs {a_1}\abs {b_1}}b_1\circ^{1}_\mu (a_1\circ^{1}_\lambda c_1))\otimes ((-1)^{\abs {a_2}\abs {b_2}}[{b_2}_{\mu}[{a_2}_{\lambda}c_2]^{2}]^{2})\\
&-(a_1\circ^{1}_\lambda (b_1\circ^{1}_\mu c_1))\otimes [{a_2}_{\lambda}[{b_2}_{\mu}c_2]^{2}]^{2},\\
h_0(Q1)&=([{a_1}_{\lambda}b_1]^{1}\circ^{1}_{\lambda+\mu} c_1)\otimes [{(a_2\circ^{2}_\lambda b_2)}_{\lambda+\mu}c_2]^{2}\\
&+((-1)^{\abs {a_1}\abs {b_1}}b_1\circ^{1}_\mu [{a_1}_{\lambda}c_1]^{1})\otimes ((-1)^{\abs {a_2}\abs {b_2}}[{b_2}_{\mu}(a_2\circ^{2}_\lambda c_2)]^{2})\\
&-(a_1\circ^{1}_\lambda [{b_1}_{\mu}c_1]^{1})\otimes [{a_2}_{\lambda}(b_2\circ^{2}_\mu c_2)]^{2},\\
h_0(Q2)&=[{(a_1\circ^{1}_\lambda b_1)}_{\lambda+\mu}c_1]^{1}\otimes ([{a_2}_{\lambda}b_2]^{2}\circ^{2}_{\lambda+\mu} c_2)\\
&+((-1)^{\abs {a_1}\abs {b_1}}[{b_1}_{\mu}(a_1\circ^{1}_\lambda c_1)]^{1})\otimes ((-1)^{\abs {a_2}\abs {b_2}}b_2\circ^{2}_\mu [{a_2}_{\lambda}c_2]^{2})\\
&-[{a_1}_{\lambda}(b_1\circ^{1}_\mu c_1)]^{1}\otimes (a_2\circ^{2}_\lambda [{b_2}_{\mu}c_2]^{2}).
\end{align*}

Since $(\mathcal{L}_1 ,\circ^{1}_{\lambda})$, $(\mathcal{L}_2 ,\circ^{2}_{\lambda})$ are commutative associative conformal superalgebras and
$(\mathcal{L}_1,[\cdot_{\lambda}\cdot]^1)$, $(\mathcal{L}_2,[\cdot_{\lambda}\cdot]^2)$ are Lie conformal superalgebras, $(Y1)$ and $(Y2)$ are zero.
By Eq.~\eqref{eqh}, we obtain
\begin{align*}
&((-1)^{\abs {a_1}\abs {b_1}}b_1\circ^{1}_\mu [{a_1}_{\lambda}c_1]^{1})\otimes ((-1)^{\abs {a_2}\abs {b_2}}[{b_2}_{\mu}(a_2\circ^{2}_\lambda c_2)]^{2})\\
&=(a_1\circ^{1}_\lambda [{b_1}_{\mu}c_1]^{1}+[{a_1}_{\lambda}b_1]^{1}\circ^{1}_{\lambda+\mu}c_1)\otimes ((-1)^{\abs {a_2}\abs {b_2}}[{b_2}_{\mu}(a_2\circ^{2}_\lambda c_2)]^{2}),\\
&((-1)^{\abs {a_1}\abs {b_1}}[{b_1}_{\mu}(a_1\circ^{1}_\lambda c_1)]^{1})\otimes ((-1)^{\abs {a_2}\abs {b_2}}b_2\circ^{2}_\mu [{a_2}_{\lambda}c_2]^{2})\\
&=((-1)^{\abs {a_1}\abs {b_1}}[{b_1}_{\mu}(a_1\circ^{1}_\lambda c_1)]^{1})\otimes (a_2\circ^{2}_\lambda [{b_2}_{\mu}c_2]^{2}+[{a_2}_{\lambda}b_2]^{2}\circ^{2}_{\lambda+\mu}c_2).
\end{align*}
Hence, we get
\begin{align*}
h_0(Q1)&=(a_1\circ^{1}_\lambda [{b_1}_{\mu}c_1]^{1})\otimes ((-1)^{\abs {a_2}\abs {b_2}}[{b_2}_{\mu}(a_2\circ^{2}_\lambda c_2)]^{2}-[{a_2}_{\lambda}(b_2\circ^{2}_\mu c_2)]^{2})\\
&+([{a_1}_{\lambda}b_1]^{1}\circ^{1}_{\lambda+\mu} c_1)\otimes ([{(a_2\circ^{2}_\lambda b_2)}_{\lambda+\mu}c_2]^{2}+(-1)^{\abs {a_2}\abs {b_2}}[{b_2}_{\mu}(a_2\circ^{2}_\lambda c_2)]^{2}),\\
h_0(Q2)&=((-1)^{\abs {a_1}\abs {b_1}}[{b_1}_{\mu}(a_1\circ^{1}_\lambda c_1)]^{1}-[{a_1}_{\lambda}(b_1\circ^{1}_\mu c_1)]^{1})\otimes (a_2\circ^{2}_\lambda [{b_2}_{\mu}c_2]^{2})\\
&+([{(a_1\circ^{1}_\lambda b_1)}_{\lambda+\mu}c_1]^{1}+(-1)^{\abs {a_1}\abs {b_1}}[{b_1}_{\mu}(a_1\circ^{1}_\lambda c_1)]^{1})\otimes ([{a_2}_{\lambda}b_2]^{2}\circ^{2}_{\lambda+\mu} c_2).
\end{align*}
By Eqs.~\eqref{rule3} and \eqref{rule5}, we have
\begin{align*}
&h_0(Q1)
=-2(a_1\circ^{1}_\lambda [{b_1}_{\mu}c_1]^{1})\otimes ([{a_2}_{\lambda}b_2]^{2}\circ^2_{\lambda+\mu}c_2)
+2([{a_1}_{\lambda}b_1]^{1}\circ^{1}_{\lambda+\mu} c_1)\otimes (a_2\circ^{2}_\lambda [{b_2}_{\mu}c_2]^{2}),\\
&h_0(Q2)
=-2([{a_1}_{\lambda}b_1]^{1}\circ^1_{\lambda+\mu}c_1)\otimes (a_2\circ^{2}_\lambda [{b_2}_{\mu}c_2]^{2})
+2(a_1\circ^{1}_\lambda [{b_1}_{\mu}c_1]^{1})\otimes ([{a_2}_{\lambda}b_2]^{2}\circ^{2}_{\lambda+\mu} c_2).
\end{align*}
Hence $(Q1)+(Q2)=0$. Therefore, $(\mathcal L,[\cdot_\lambda\cdot])$ is a Lie conformal superalgebra.

Finally, we show that the relation \eqref{rule5} holds. We compute separately 
\begin{align*}
&[(a_1\otimes a_2)_\lambda ((b_1\otimes b_2)\circ_\mu (c_1\otimes c_2))]
=(-1)^{\abs {b_2}\abs{c_1}}[(a_1\otimes a_2)_\lambda ((b_1\circ^{1}_\mu c_1)\otimes (b_2\circ^{2}_\mu c_2))]\\
&=(-1)^{\abs {b_2}\abs{c_1}+\abs{a_2}(\abs{b_1}+\abs{c_1})}\big([{a_1}_{\lambda}(b_1\circ^{1}_\mu c_1)]^{1}\otimes (a_2\circ^{2}_\lambda (b_2\circ^{2}_\mu c_2))\big.\\
&\quad \big.+(a_1\circ^{1}_\lambda (b_1\circ^{1}_\mu c_1))\otimes [{a_2}_{\lambda}(b_2\circ^{2}_\mu c_2)]^{2}\big),\\
&[(b_1\otimes b_2)_\mu ((a_1\otimes a_2)\circ_\lambda (c_1\otimes c_2))]
=(-1)^{\abs {a_2}\abs{c_1}}[(b_1\otimes b_2)_\mu ((a_1\circ^{1}_\lambda c_1)\otimes (a_2\circ^{2}_\lambda c_2))]\\
&=(-1)^{\abs {a_2}\abs{c_1}+\abs{b_2}(\abs{a_1}+\abs{c_1})}\big([{b_1}_{\mu}(a_1\circ^{1}_\lambda c_1)]^{1}\otimes (b_2\circ^{2}_\mu (a_2\circ^{2}_\lambda c_2))\big.\\
&\quad \big.+(b_1\circ^{1}_\mu (a_1\circ^{1}_\lambda c_1))\otimes [{b_2}_{\mu}(a_2\circ^{2}_\lambda c_2)]^{2}\big).
\end{align*}
Then we set $s_0=(-1)^{\abs {b_2}\abs{c_1}+\abs{a_2}(\abs{b_1}+\abs{c_1})}$ and obtain 
\begin{align*}
&[(a_1\otimes a_2)_\lambda ((b_1\otimes b_2)\circ_\mu (c_1\otimes c_2))]-(-1)^{(\abs {a_1}+\abs {a_2})(\abs {b_1}+\abs {b_2})}[(b_1\otimes b_2)_\mu ((a_1\otimes a_2)\circ_\lambda (c_1\otimes c_2))]\\
&=s_0([{a_1}_{\lambda}(b_1\circ^{1}_\mu c_1)]^{1}-(-1)^{\abs {a_1}\abs {b_1}}[{b_1}_{\mu}(a_1\circ^{1}_\lambda c_1)]^{1})\otimes ((a_2\circ^{2}_\lambda b_2)\circ^{2}_{\lambda+\mu} c_2)\\
&\hphantom{=~}+s_0((a_1\circ^{1}_\lambda b_1)\circ^{1}_{\lambda+\mu} c_1)\otimes ([{a_2}_{\lambda}(b_2\circ^{2}_\mu c_2)]^{2}-(-1)^{\abs {a_2}\abs {b_2}}[{b_2}_{\mu}(a_2\circ^{2}_\lambda c_2)]^{2})\\
&=2s_0([{a_1}_{\lambda}b_1]^{1}\circ^1_{\lambda+\mu}c_1)\otimes ((a_2\circ^{2}_\lambda b_2)\circ^{2}_{\lambda+\mu} c_2)
+2s_0((a_1\circ^{1}_\lambda b_1)\circ^{1}_{\lambda+\mu} c_1)\otimes ([{a_2}_{\lambda}b_2]^{2}\circ^2_{\lambda+\mu}c_2)\\
&=(-1)^{\abs {a_2}\abs{b_1}}2\big(([{a_1}_{\lambda}b_1]^{1}\otimes (a_2\circ^{2}_\lambda b_2))\circ_{\lambda+\mu} (c_1\otimes c_2)
+((a_1\circ^{1}_\lambda b_1)\otimes [{a_2}_{\lambda}b_2]^{2})\circ_{\lambda+\mu} (c_1\otimes c_2)\big)\\
&=(-1)^{\abs {a_2}\abs{b_1}}2([{a_1}_{\lambda}b_1]^{1}\otimes (a_2\circ^{2}_\lambda b_2)+(a_1\circ^{1}_\lambda b_1)\otimes [{a_2}_{\lambda}b_2]^{2})\circ_{\lambda+\mu} (c_1\otimes c_2)\\
&=2[(a_1\otimes a_2)_\lambda (b_1\otimes b_2)]\circ_{\lambda+\mu} (c_1\otimes c_2).
\end{align*}

Altogether, $(\mathcal L,\circ_\lambda,[\cdot_\lambda\cdot])$ is a transposed Poisson conformal superalgebra.
\end{proof}

\subsection{Relation with Hom-Lie conformal superalgebras}

\begin{proposition}
   Let $(\mathcal{L} ,\circ_\lambda,[\cdot_{\lambda}\cdot])$ be a TPCSA. For any  $h\in \mathcal{L}_{\bar{0}}$, define a linear map
   $\alpha_h$ : $\mathcal{L}  \rightarrow \mathcal{L}$ by
   \begin{align*}
      \alpha_h(x):=(h\circ_\lambda x)|_{\lambda=0}=h\circ_{(0)} x, \quad \forall ~x\in\mathcal{L}.
   \end{align*}
   Then $(\mathcal{L} ,[\cdot_{\lambda}\cdot],\alpha_h)$ is a Hom-Lie conformal superalgebra.
\end{proposition}

\begin{proof}
   Since $h$ is even, $\alpha_h$ is an even linear map.  Moreover,
   the axioms of conformal sesquilinearity and skew-symmetry are naturally satisfied since
$(\mathcal{L},[\cdot_{\lambda}\cdot])$ is a Lie conformal superalgebra. For all $x\in\mathcal{L}$,
   \begin{align*}
      \alpha_h (\partial x)=(h\circ_\lambda \partial x)|_{\lambda=0}=(\partial+\lambda)(h\circ_\lambda x)|_{\lambda=0}=\partial(h\circ_{(0)} x)=\partial (\alpha_hx).
   \end{align*}
Hence, $\alpha_h\partial=\partial\alpha_h$. By taking $\gamma:=\lambda, ~\mu:=\lambda+\mu,~\lambda:=0$ in Eq.~\eqref{equ}, we obtain
   \begin{align*}
      [[x_{\lambda}y]_{\lambda+\mu}(h\circ_{(0)}z)]
      +(-1)^{\abs {x}\abs{y}}[(h\circ_{(0)}y)_{\mu}[x_{\lambda}z]]=
      [(h\circ_{(0)}x)_{\lambda}[y_{\mu}z]].
   \end{align*}
 This is equivalent to
   \begin{align*}
        [[x_{\lambda}y]_{\lambda+\mu}\alpha_h(z)]+(-1)^{\abs {x}\abs{y}}[\alpha_h(y)_{\mu}[x_{\lambda}z]]=[\alpha_h(x)_{\lambda}[y_{\mu}z]],
   \end{align*}
which is exactly the Hom-Jacobi identity. Therefore, $(\mathcal{L} ,[\cdot_{\lambda}\cdot],\alpha_h)$ is a Hom-Lie conformal superalgebra.
\end{proof}

\subsection{Compatibility conditions between PCSAs and TPCSAs}

\begin{proposition}
   Let $(\mathcal{L} ,\circ_\lambda)$ be a commutative associative conformal superalgebra and $(\mathcal{L} ,[\cdot_{\lambda}\cdot])$ be a Lie conformal superalgebra.
   Then $(\mathcal{L} ,\circ_\lambda,[\cdot_{\lambda}\cdot])$ is both a Poisson conformal superalgebra and a transposed Poisson conformal superalgebra if and only if
   \begin{align*}
      a\circ_{\lambda}[b_{\mu}c]=[b_{\mu-\lambda}a]\circ_{\mu}c=[(a\circ_{\lambda}b)_{\lambda+\mu}c]=0, \quad \forall ~a,b,c\in \mathcal{L}.
   \end{align*}
\end{proposition}

\begin{proof}
   The ``if" part can be easily verified. For the ``only if" part, let $a,b,c\in \mathcal{L}$. By Eq.~\eqref{rule5}, we have
   \begin{align*}
      2[a_{\lambda}b]\circ_{\mu}c=[a_{{\lambda}}(b\circ_{\mu-\lambda}c)]-(-1)^{\abs {a}\abs{b}}[b_{\mu-\lambda}(a\circ_{\lambda}c)].
   \end{align*}
   By Eq.~\eqref{rule1}, we obtain
   \begin{align*}
      &[a_{\lambda}(b\circ_{\mu-\lambda}c)]=[a_{\lambda}b]\circ_{\mu}c+(-1)^{\abs {a}\abs{b}}b\circ_{\mu-\lambda}[a_{\lambda}c],  \\
      &[b_{\mu-\lambda}(a\circ_{\lambda}c)]=[b_{\mu-\lambda}a]\circ_{\mu}c+(-1)^{\abs {a}\abs{b}}a\circ_{\lambda}[b_{\mu-\lambda}c].
   \end{align*}
   Thus, we have
   \begin{align*}
      [a_{\lambda}b]\circ_{\mu}c+a\circ_{\lambda}[b_{\mu-\lambda}c]-(-1)^{\abs {a}\abs{b}}b\circ_{\mu-\lambda}[a_{\lambda}c]+(-1)^{\abs {a}\abs{b}}[b_{\mu-\lambda}a]\circ_{\mu}c=0.
   \end{align*}
   By Eq.~\eqref{eqw}, we get
   \begin{align*}
      [a_{\lambda}b]\circ_{\mu}c+a\circ_{\lambda}[b_{\mu-\lambda}c]-(-1)^{\abs {a}\abs{b}}b\circ_{\mu-\lambda}[a_{\lambda}c]=0.
   \end{align*}
   Hence, $[b_{\mu-\lambda}a]\circ_{\mu}c=0$. From this and Eq.~\eqref{eqh}, we get $a\circ_{\lambda}[b_{\mu}c]=(-1)^{\abs {a}\abs{b}}b \circ_{\mu} [a _\lambda c]$. By Eq.~\eqref{rule3}, we have
   \begin{align*}
      2a\circ_{\lambda}[b_{\mu}c]=[(a\circ_{\lambda}b)_{\lambda+\mu}c]+(-1)^{\abs {a}\abs{b}}[b_{\mu}(a\circ_{\lambda}c)].	
   \end{align*}
   By Eqs.~\eqref{rule1} and \eqref{rule6}, we obtain
   \begin{align*}
      &[(a\circ_{\lambda}b)_{\lambda+\mu}c] = (-1)^{\abs {a}\abs{b}}b \circ_{\mu} [a _\lambda c] + a \circ_\lambda [b _{\mu} c],  \\
      &[b_{\mu}(a\circ_{\lambda}c)] = [b_{\mu}a]\circ_{\lambda+\mu}c+(-1)^{\abs {a}\abs{b}}a\circ_{\lambda}[b_{\mu}c].
   \end{align*}
   Thus, $a\circ_{\lambda}[b_{\mu}c]=-(-1)^{\abs {a}\abs{b}}[b_{\mu}a]\circ_{\lambda+\mu}c=0$, $[(a\circ_{\lambda}b)_{\lambda+\mu}c]=2a\circ_{\lambda}[b_{\mu}c]=0$.
\end{proof}

\section{Constructions of transposed Poisson conformal superalgebras}\label{sec4}
In this section, we present several constructions of transposed Poisson conformal superalgebras (TPCSAs). We first show how to obtain a new TPCSA from a given one by introducing a suitable binary operation. We then introduce Novikov-Poisson conformal superalgebras, pre-Lie commutative conformal superalgebras, differential Novikov-Poisson conformal superalgebras, and pre-Lie Poisson conformal superalgebras, which are the super counterparts of the corresponding conformal algebras considered in \cite{YF2026}. Finally, we construct several classes of TPCSAs from these structures.

\begin{proposition}
   Let $(\mathcal{L} ,\circ_\lambda,[\cdot_{\lambda}\cdot])$ be a TPCSA. For any  $h\in \mathcal{L}_{\bar{0}}$, define a new binary operation $[\cdot_{\lambda}\cdot]^h$ on $\mathcal{L}$ by
   \begin{align*}
      [x_{\lambda}y]^h:=(h\circ_{\mu}[x_{\lambda}y])|_{\mu=0}=h\circ_{(0)}[x_{\lambda}y], \quad \forall ~x,y\in \mathcal{L}.
   \end{align*}
   Then $(\mathcal{L} ,\circ_\lambda, [\cdot_{\lambda}\cdot]^h)$ is a TPCSA.
\end{proposition}

\begin{proof}
Let $x,y,z\in \mathcal{L},h\in \mathcal{L}_{\bar{0}}$.
Since $h$ is even, the operation $[\cdot_\lambda\cdot]^h$ is even.
We first show that $(\mathcal{L}, [\cdot_{\lambda}\cdot]^h)$ is a Lie conformal superalgebra.

   (a) {\bf Conformal sesquilinearity:}
   \begin{align*}
      [(\partial x)_{\lambda}y]^h=h\circ_{(0)}[(\partial x)_{\lambda}y]=h\circ_{(0)}(-\lambda[x_{\lambda}y])=-\lambda (h\circ_{(0)}[x_{\lambda}y])=-\lambda [x_{\lambda}y]^h.
   \end{align*}

   (b) {\bf Skew-symmetry:}
   \begin{align*}
      [x_{\lambda}y]^h&=h\circ_{(0)}[x_{\lambda}y]=h\circ_{(0)}(-(-1)^{\abs {x}\abs{y}}[y_{-\partial-\lambda}x])\\
      &=-(-1)^{\abs {x}\abs{y}}h\circ_{(0)}[y_{-\partial-\lambda}x]=-(-1)^{\abs {x}\abs{y}}[y_{-\partial-\lambda}x]^h.
   \end{align*}

   (c) {\bf Jacobi identity:} By taking $\gamma:=\lambda, ~\mu:=\lambda+\mu,~\lambda:=0$ in Eq.~\eqref{eqs}, we obtain
   \begin{align*}
      &[(h\circ_{(0)}[x_{\lambda}y])_{\lambda+\mu}z]
      +(-1)^{\abs {x}\abs{y}}[y_{\mu}(h\circ_{(0)}[x_{\lambda}z])]
      -[x_{\lambda}(h\circ_{(0)}[y_{\mu}z])]=0
   \end{align*}
   Thus, we have
   $$h\circ_{(0)}([(h\circ_{(0)}[x_{\lambda}y])_{\lambda+\mu}z]
      +(-1)^{\abs {x}\abs{y}}[y_{\mu}(h\circ_{(0)}[x_{\lambda}z])]
      -[x_{\lambda}(h\circ_{(0)}[y_{\mu}z])])=0.$$
    That is,
   \begin{align*}
      [{[x_{\lambda}y]^h}_{\lambda+\mu}z]^h+(-1)^{\abs {x}\abs{y}}[y_{\mu}{[x_{\lambda}z]^h}]^h-[x_{\lambda}{[y_{\mu}z]^h}]^h=0.
   \end{align*}

   Next we show that transposed conformal super-Leibniz rule holds. By Eq.~\eqref{rule5}, we obtain
     \begin{align*}
      & 2([x_{\lambda}y]^h\circ_{\mu}z)=2((h\circ_{(0)}[x_{\lambda}y])\circ_{\mu}z)
      =h\circ_{(0)}(2([x_{\lambda}y]\circ_{\mu}z))\\
      &= h\circ_{(0)}([x_{{\lambda}}(y\circ_{\mu-\lambda}z)]-(-1)^{\abs {x}\abs{y}}[y_{\mu-\lambda}(x\circ_{\lambda}z)])\\
      &=h\circ_{(0)}[x_{{\lambda}}(y\circ_{\mu-\lambda}z)]-(-1)^{\abs {x}\abs{y}}h\circ_{(0)}[y_{\mu-\lambda}(x\circ_{\lambda}z)]  \\
      &= [x_{{\lambda}}(y\circ_{\mu-\lambda}z)]^h-(-1)^{\abs {x}\abs{y}}[y_{\mu-\lambda}(x\circ_{\lambda}z)]^h.
     \end{align*}
Therefore, $(\mathcal{L} ,\circ_\lambda, [\cdot_{\lambda}\cdot]^h)$ is a transposed Poisson conformal superalgebra.
\end{proof}

\begin{definition}
   A {\textbf{Novikov-Poisson conformal superalgebra}} $\mathcal{N}=\mathcal{N}_{\bar{0}}\oplus \mathcal{N}_{\bar{1}}$ is a $\Ztwo$-graded $\mathbb{C}[\partial]$-module endowed with two $\mathbb{C}$-bilinear maps $\mathcal{N}  \otimes \mathcal{N}  \rightarrow \mathcal{N} [\lambda]$, denoted by $a\otimes b\mapsto a \circ_{\lambda}b$ and $a\otimes b\mapsto a\ast_\lambda b$, respectively,
   such that $(\mathcal{N} ,\circ_\lambda)$ is a {commutative associative conformal superalgebra}, $(\mathcal{N} ,\ast_\lambda)$ is a {Novikov conformal superalgebra} and the following equations hold for all $a,b,c \in \mathcal{N}$:
   \begin{align}
      &(a\circ_\lambda b)\ast_{\lambda+\mu}c=a\circ_\lambda(b\ast_\mu c), \label{f7}\\
      &(a\ast_\lambda b)\circ_{\lambda+\mu}c-a\ast_\lambda(b\circ_\mu c)=(-1)^{\abs a\abs b}((b\ast_\mu a)\circ_{\lambda+\mu}c-b\ast_\mu(a\circ_\lambda c)). \label{h8}
   \end{align}
\end{definition}

\begin{remark}\label{rem2}
   In analogy with Remark \ref{rem1}, the following properties always hold in a Novikov-Poisson conformal superalgebra $(\mathcal{N},\circ_{\lambda},\ast_{\lambda})$:
   \begin{align*}
         a\circ_{\lambda}(b\ast_{-\partial-\mu}c) & =(a\circ_{\lambda}b)\ast_{-\partial-\mu}c, \\
         a\circ_{-\partial-\lambda}(b\ast_{\mu}c) & =(a\circ_{-\partial-\mu}b)\ast_{-\partial+\mu-\lambda}c, \\
         a\circ_{-\partial-\lambda}(b\ast_{-\partial-\mu}c) & =(a\circ_{-\partial+\mu-\lambda}b)\ast_{-\partial-\mu}c.
   \end{align*}
\end{remark}

Taking the commutator in a Novikov-Poisson conformal superalgebra, we obtain
\begin{theorem}\label{The1.2}
   Let $(\mathcal{N} ,\circ_\lambda,\ast_\lambda)$ be a Novikov-Poisson conformal superalgebra. Define
   \begin{align}
      [a_{\lambda}b]=a\ast_{\lambda} b-(-1)^{\abs a\abs b}(b\ast_{-\partial-\lambda} a), \quad \forall ~a,b \in \mathcal{N}.
   \end{align}
Then $(\mathcal{N} ,\circ_\lambda,[\cdot_\lambda\cdot])$ is a TPCSA.
\end{theorem}

\begin{proof} By Proposition \ref{propo1}, $(\mathcal{N},[\cdot_\lambda\cdot])$ is a Lie conformal superalgebra. For all $a,b,c\in \mathcal{N}$, we have
   \begin{align*}
      &[a_{\lambda}b]\circ_{\mu}c=(a\ast_{\lambda} b)\circ_{\mu}c-(-1)^{\abs a\abs b}(b\ast_{-\partial-\lambda} a)\circ_{\mu}c, \\
      &[a_{{\lambda}}(b\circ_{\mu-\lambda}c)]=a\ast_{\lambda} (b\circ_{\mu-\lambda}c)-(-1)^{\abs a(\abs b+\abs c)}(b\circ_{\mu-\lambda}c)\ast_{-\partial-\lambda} a, \\
      &[b_{\mu-\lambda}(a\circ_{\lambda}c)]=b\ast_{\mu-\lambda} (a\circ_{\lambda}c)-(-1)^{\abs b(\abs a+\abs c)}(a\circ_{\lambda}c)\ast_{-\partial-\mu+\lambda} b.
   \end{align*}
   By Remark \ref{rem2}, we obtain
   \begin{align*}
      (a\ast_{\lambda} b)\circ_{\mu}c&=(-1)^{\abs c(\abs a+\abs b)}c\circ_{-\partial-\mu}(a\ast_{\lambda} b)\\
      &=(-1)^{\abs c(\abs a+\abs b)}(c\circ_{-\partial-\lambda}a)\ast_{-\partial+\lambda-\mu}b=(-1)^{\abs b\abs c}(a\circ_{\lambda}c)\ast_{-\partial-\mu+\lambda} b, \\
      (b\ast_{-\partial-\lambda} a)\circ_{\mu}c&=(-1)^{\abs c(\abs a+\abs b)}c\circ_{-\partial-\mu}(b\ast_{-\partial-\lambda} a)\\
      &=(-1)^{\abs c(\abs a+\abs b)}(c\circ_{-\partial+\lambda-\mu}b)\ast_{-\partial-\lambda}a=(-1)^{\abs a\abs c}(b\circ_{\mu-\lambda}c)\ast_{-\partial-\lambda}a.
   \end{align*}
   Thus, we have
   \begin{align*}
      [a_{{\lambda}}(b\circ_{\mu-\lambda}c)]-(-1)^{\abs a\abs b}[b_{\mu-\lambda}(a\circ_{\lambda}c)]=a\ast_{\lambda} (b\circ_{\mu-\lambda}c)-(-1)^{\abs a\abs b}b\ast_{\mu-\lambda} (a\circ_{\lambda}c)+[a_{\lambda}b]\circ_{\mu}c.
   \end{align*}
   By taking $\lambda:=\lambda,~\mu:=\mu-\lambda$ in Eq.~\eqref{h8}, we get
   \begin{align*}
      (a\ast_\lambda b)\circ_{\mu}c-a\ast_\lambda(b\circ_{\mu-\lambda} c)=(-1)^{\abs a\abs b}((b\ast_{\mu-\lambda} a)\circ_{\mu}c-b\ast_{\mu-\lambda}(a\circ_\lambda c)).
   \end{align*}
   Since $(b\ast_{-\partial-\lambda} a)\circ_{\mu}c=(b\ast_{\mu-\lambda} a)\circ_{\mu}c$, it follows further that
   \begin{align*}
    (a\ast_\lambda b)\circ_{\mu}c-(-1)^{\abs a\abs b}(b\ast_{\mu-\lambda} a)\circ_{\mu}c&=a\ast_\lambda(b\circ_{\mu-\lambda} c)-(-1)^{\abs a\abs b}b\ast_{\mu-\lambda}(a\circ_\lambda c)\\
    &=[a_{\lambda}b]\circ_{\mu}c.
   \end{align*}
   Combining the above equalities, we arrive at
   \begin{align*}
      [a_{{\lambda}}(b\circ_{\mu-\lambda}c)]-(-1)^{\abs a\abs b}[b_{\mu-\lambda}(a\circ_{\lambda}c)]=2[a_{\lambda}b]\circ_{\mu}c,
   \end{align*}
   so relation \eqref{rule5} holds. This completes the proof.
\end{proof}

\begin{lemma}\label{Lem1.1}
   Let $(\mathcal{N} ,\circ_\lambda)$ be a commutative associative conformal superalgebra and $D$ be an even derivation. Define a binary operation $\ast_\lambda$ on $\mathcal{N}$ by
   \begin{align}
      a\ast_\lambda b=a\circ_\lambda D(b), \quad \forall~a,b \in \mathcal{N}.
   \end{align}
Then the following conclusions hold:
\begin{itemize}
\item[(1)] $(\mathcal{N} ,\ast_\lambda)$ is a Novikov conformal superalgebra,
\item[(2)] $(\mathcal{N} ,\circ_\lambda,\ast_\lambda)$ is a Novikov-Poisson conformal superalgebra.
\end{itemize}


\end{lemma}

\begin{proof}
(1) Let $a,b,c \in \mathcal{N}$.  We first show that $(\mathcal{N} ,\ast_\lambda)$ is a conformal superalgebra.
   \begin{align*}
      &(\partial a)\ast_\lambda b=(\partial a)\circ_\lambda D(b)=-\lambda (a\circ_\lambda D(b))=-\lambda(a\ast_\lambda b), \\
      &a\ast_\lambda (\partial b)=a\circ_\lambda D(\partial b)=a\circ_\lambda (\partial D(b))=(\partial+\lambda) (a\circ_\lambda D(b))=(\partial+\lambda)(a\ast_\lambda b).
   \end{align*}
Next we show that $(\mathcal{N} ,\ast_\lambda)$ is a left-symmetric conformal superalgebra. Since
   \begin{align*}
      &(a\ast_\lambda b)\ast_{\lambda+\mu}c-a\ast_\lambda(b\ast_\mu c)=(a\circ_\lambda D(b))\circ_{\lambda+\mu} D(c)-a\circ_\lambda D(b\circ_\mu D(c))  \\
      &=a\circ_\lambda (D(b)\circ_{\mu} D(c))-a\circ_\lambda (D(b)\circ_\mu D(c))-a\circ_\lambda (b\circ_\mu D^2(c))=-a\circ_\lambda (b\circ_\mu D^2(c)),
   \end{align*}
we similarly have
   \begin{align*}
        (b\ast_\mu a)\ast_{\lambda+\mu}c-b\ast_\mu(a\ast_\lambda c)=-b\circ_\mu (a\circ_\lambda D^2(c)).
   \end{align*}
By Remark \ref{rem3}, we obtain $a\circ_\lambda (b\circ_\mu D^2(c))=(-1)^{\abs a\abs b}b\circ_\mu (a\circ_\lambda D^2(c)).$ Thus, we get
$$(a\ast_\lambda b)\ast_{\lambda+\mu}c-a\ast_\lambda(b\ast_\mu c)=(-1)^{\abs a\abs b}((b\ast_\mu a)\ast_{\lambda+\mu}c-b\ast_\mu(a\ast_\lambda c)).$$
  Finally, we show that $(\mathcal{N} ,\ast_\lambda)$ is a Novikov conformal superalgebra. Since
   \begin{align*}
      &(a\ast_\lambda c)\ast_{-\mu-\partial}b=(a\circ_\lambda D(c))\circ_{-\mu-\partial}D(b)=(-1)^{\abs b(\abs a+\abs c)}D(b)\circ_{\mu}(a\circ_\lambda D(c)) \\
      &=(-1)^{\abs b\abs c}a\circ_\lambda(D(b)\circ_{\mu} D(c))=(-1)^{\abs b\abs c}(a\circ_\lambda D(b))\circ_{\lambda+\mu} D(c)=(-1)^{\abs b\abs c}(a\ast_\lambda b)\ast_{\lambda+\mu}c.
   \end{align*}
Hence, $(\mathcal{N} ,\ast_\lambda)$ is a Novikov conformal superalgebra.

(2) Let $a,b,c \in \mathcal{N}$. Relation \eqref{f7} holds, since
   \begin{align*}
      (a\circ_\lambda b)\ast_{\lambda+\mu}c=(a\circ_\lambda b)\circ_{\lambda+\mu}D(c)=a\circ_\lambda(b\circ_{\mu}D(c))=a\circ_\lambda(b\ast_{\mu}c).
   \end{align*}
To show relation \eqref{h8} holds, we compute
   \begin{align*}
      &(a\ast_\lambda b)\circ_{\lambda+\mu}c-a\ast_\lambda(b\circ_\mu c)=(a\circ_\lambda D(b))\circ_{\lambda+\mu}c-a\circ_\lambda D(b\circ_\mu c) \\
      &=a\circ_\lambda (D(b)\circ_{\mu}c)-a\circ_\lambda (D(b)\circ_{\mu}c)-a\circ_\lambda (b\circ_{\mu}D(c))=-a\circ_\lambda (b\circ_{\mu}D(c)),
   \end{align*}
and we similarly have
 $$(b\ast_\mu a)\circ_{\lambda+\mu}c-b\ast_\mu(a\circ_\lambda c)=-b\circ_\mu (a\circ_{\lambda}D(c)).$$
By Remark \ref{rem3}, we get $a\circ_\lambda (b\circ_{\mu}D(c))=(-1)^{\abs a\abs b}b\circ_\mu (a\circ_{\lambda}D(c))$. Thus, we obtain
   \begin{align*}
      (a\ast_\lambda b)\circ_{\lambda+\mu}c-a\ast_\lambda(b\circ_\mu c)=(-1)^{\abs a\abs b}((b\ast_\mu a)\circ_{\lambda+\mu}c-b\ast_\mu(a\circ_\lambda c)).
   \end{align*}

Altogether, $(\mathcal{N} ,\circ_\lambda,\ast_\lambda)$ is a Novikov-Poisson conformal superalgebra.
\end{proof}

As a consequence of Theorem \ref{The1.2} and Lemma \ref{Lem1.1}, we obtain 

\begin{corollary}
   Let $(\mathcal{N} ,\circ_\lambda)$ be a commutative associative conformal superalgebra and $D$ be an even derivation. Then $(\mathcal{N} ,\circ_\lambda,[\cdot_{\lambda}\cdot])$ is a TPCSA, where
   \begin{align}
      [a_{\lambda}b]=a\circ_\lambda D(b)-(-1)^{\abs a\abs b}b\circ_{-\partial-\lambda} D(a), \quad \forall ~a,b \in \mathcal{N}.
   \end{align}
\end{corollary}

\begin{definition}
   A {\bf pre-Lie commutative conformal superalgebra}  $\mathcal{L}=\mathcal{L}_{\bar{0}}\oplus \mathcal{L}_{\bar{1}}$ is a $\Ztwo$-graded $\mathbb{C}[\partial]$-module $\mathcal{L}$ endowed with two $\mathbb{C}$-bilinear maps $\mathcal{L}  \otimes \mathcal{L}  \rightarrow \mathcal{L} [\lambda]$, denoted by $a\otimes b\mapsto a \circ_{\lambda}b$ and 
    $a\otimes b\mapsto a \ast_\lambda b$, respectively, such that $(\mathcal{L} ,\circ_\lambda)$ is a {commutative associative conformal superalgebra}, $(\mathcal{L} ,\ast_\lambda)$ is a {left-symmetric conformal superalgebra} and the following equation holds for all $a,b,c \in \mathcal{L}$:
   \begin{align}\label{hh1}
      a\ast_\lambda(b\circ_\mu c)=(a\ast_\lambda b)\circ_{\lambda+\mu}c+(-1)^{\abs a\abs b}b\circ_\mu (a\ast_\lambda c).
   \end{align}
\end{definition}

\begin{example}
   Let $(\mathcal{L} ,\circ_\lambda)$ be a commutative associative conformal superalgebra and $D$ be an even derivation.
   It is straightforward to verify that $(\mathcal{L} ,\circ_\lambda,\ast_\lambda)$ is a pre-Lie commutative conformal superalgebra,
   where $(\mathcal{L} ,\ast_\lambda)$ is defined by $a\ast_\lambda b=a\circ_\lambda D(b)$ for all $a,b \in \mathcal{L}$.
\end{example}

A {\bf differential Novikov-Poisson conformal superalgebra} is a triple $(\mathcal{N} ,\circ_\lambda,\ast_\lambda)$ such that
$(\mathcal{N} ,\circ_\lambda)$ is a {commutative associative conformal superalgebra}, $(\mathcal{N} ,\ast_\lambda)$ is a {Novikov conformal superalgebra},
and Eqs.~\eqref{f7}, \eqref{h8} and \eqref{hh1} hold.

\begin{definition}
   A {\textbf{pre-Lie Poisson conformal superalgebra}} $\mathcal{L}=\mathcal{L}_{\bar{0}}\oplus \mathcal{L}_{\bar{1}}$ is a $\Ztwo$-graded $\mathbb{C}[\partial]$-module endowed with two $\mathbb{C}$-bilinear maps $\mathcal{L}  \otimes \mathcal{L}  \rightarrow \mathcal{L} [\lambda]$, denoted by $a\otimes b\mapsto a \circ_{\lambda}b$ and 
    $a\otimes b\mapsto a \ast_\lambda b$, respectively, such that $(\mathcal{L} ,\circ_\lambda)$ is a {commutative associative conformal superalgebra}, $(\mathcal{L} ,\ast_\lambda)$ is a {left-symmetric conformal superalgebra} and the following equations hold for all $a,b,c \in \mathcal{L}$:
   \begin{align}
      &(a\circ_\lambda b)\ast_{\lambda+\mu}c=a\circ_\lambda(b\ast_\mu c),\label{eq:4-7}\\
      &(a\ast_\lambda b)\circ_{\lambda+\mu}c-a\ast_\lambda(b\circ_\mu c)=(-1)^{\abs a\abs b}((b\ast_\mu a)\circ_{\lambda+\mu}c-b\ast_\mu(a\circ_\lambda c)). \label{eq:4-8}
   \end{align}
\end{definition}

\begin{lemma}\label{Lem1.2}
The relationship among these notions can be described as follows:

(1) A Novikov-Poisson conformal superalgebra is a pre-Lie Poisson conformal superalgebra.

(2) A pre-Lie commutative conformal superalgebra $(\mathcal{L} ,\circ_\lambda,\ast_\lambda)$ satisfying Eq.~\eqref{f7} is a pre-Lie Poisson conformal superalgebra.

(3) A differential Novikov-Poisson conformal superalgebra is a pre-Lie commutative conformal superalgebra satisfying Eq.~\eqref{f7} and hence is pre-Lie Poisson conformal superalgebra.
\end{lemma}

\begin{proof}
    The proofs of (1) and (3) are straightforward. It suffices to prove (2). Let $a,b,c \in \mathcal{L}$. By Eqs.~\eqref{f7} and \eqref{hh1}, we have
   \begin{align*}
        &(a\ast_\lambda b)\circ_{\lambda+\mu}c-a\ast_\lambda(b\circ_\mu c)=-(-1)^{\abs a\abs b}b\circ_\mu (a\ast_\lambda c)=-(-1)^{\abs a\abs b}(b\circ_{\mu} a)\ast_{\lambda+\mu}c,\\
        &(b\ast_\mu a)\circ_{\lambda+\mu}c-b\ast_\mu(a\circ_\lambda c)=-(-1)^{\abs a\abs b}a\circ_\lambda (b\ast_\mu c)=-(-1)^{\abs a\abs b}(a\circ_\lambda b)\ast_{\lambda+\mu}c.
   \end{align*}
Observe that
   \begin{align*}
        (a\circ_\lambda b)\ast_{\lambda+\mu}c=(-1)^{\abs a\abs b}(b\circ_{-\partial-\lambda} a)\ast_{\lambda+\mu}c=(-1)^{\abs a\abs b}(b\circ_{\mu} a)\ast_{\lambda+\mu}c.
   \end{align*}
Thus, we obtain
   \begin{align*}
       (a\ast_\lambda b)\circ_{\lambda+\mu}c-a\ast_\lambda(b\circ_\mu c)=(-1)^{\abs a\abs b}((b\ast_\mu a)\circ_{\lambda+\mu}c-b\ast_\mu(a\circ_\lambda c)).
   \end{align*}
Hence, $(\mathcal{L} ,\circ_\lambda,\ast_\lambda)$ is a pre-Lie Poisson conformal superalgebra.
\end{proof}

It follows from Proposition \ref{propo1} that the commutator of a left-symmetric conformal superalgebra is a Lie conformal superalgebra.
Applying an argument similar to that used in the proof of Theorem \ref{The1.2}, we obtain the following proposition.

\begin{proposition}\label{Pro2}
   Let $(\mathcal{L} ,\circ_\lambda,\ast_\lambda)$ be a pre-Lie Poisson conformal superalgebra. Define
   \begin{align}
      [a_{\lambda}b]=a\ast_{\lambda} b-(-1)^{\abs a\abs b}(b\ast_{-\partial-\lambda} a), \quad \forall ~a,b \in \mathcal{L}.
   \end{align}
Then $(\mathcal{L} ,\circ_\lambda,[\cdot_\lambda\cdot])$ is a TPCSA.
\end{proposition}

Combining Proposition \ref{Pro2} and Lemma \ref{Lem1.2}, we obtain 
\begin{corollary}
   Let $(\mathcal{L} ,\circ_\lambda,\ast_\lambda)$ be a pre-Lie commutative conformal superalgebra. Suppose that Eq.~\eqref{f7} holds. Then $(\mathcal{L} ,\circ_\lambda,[\cdot_\lambda\cdot])$ is a TPCSA, where
   the binary operation $[\cdot_\lambda\cdot]$ is defined by $[a_{\lambda}b]=a\ast_{\lambda} b-(-1)^{\abs a\abs b}(b\ast_{-\partial-\lambda} a)$ for all $a,b \in \mathcal{L}$. In particular, the conclusion holds when $(\mathcal{L} ,\circ_\lambda,\ast_\lambda)$ is a differential Novikov-Poisson conformal superalgebra.
\end{corollary}


\begin{theorem}
Let $(\mathcal{L}_1 ,\circ^{1}_{\lambda},\ast^{1}_{\lambda})$ and $(\mathcal{L}_2 ,\circ^{2}_{\lambda},\ast^{2}_{\lambda})$ be two pre-Lie Poisson conformal superalgebras.
Set $$\mathcal L:=\mathcal L_1\otimes_{\mathbb C[\partial]} \mathcal L_2,\quad \abs{a\otimes b}=\abs a+\abs b$$ and the $\mathbb C[\partial]$-module structure on $\mathcal L$ is given by
\begin{align*}
\partial(a\otimes b):=(\partial a)\otimes b = a\otimes (\partial b),
\quad a\in\mathcal L_1,\, b\in\mathcal L_2 .
\end{align*}
Define two $\lambda$-products $\circ_\lambda$ and $*_\lambda$ on $\mathcal L$ respectively by
\begin{gather*}
(a_1\otimes a_2)\circ_\lambda (b_1\otimes b_2)
= (-1)^{\abs {a_2}\abs {b_1}}(a_1\circ^{1}_\lambda b_1)\otimes (a_2\circ^{2}_\lambda b_2),\\
(a_1\otimes a_2)\ast_\lambda (b_1\otimes b_2)
= (-1)^{\abs {a_2}\abs {b_1}}((a_1\ast^{1}_\lambda b_1)\otimes (a_2\circ^{2}_\lambda b_2)+(a_1\circ^{1}_\lambda b_1)\otimes (a_2\ast^{2}_\lambda b_2)),
\end{gather*}
for all $a_1,b_1 \in \mathcal{L}_1$, $a_2,b_2 \in \mathcal{L}_2$.
Then $(\mathcal L,\circ_\lambda,\ast_\lambda)$ is a pre-Lie Poisson conformal superalgebra.
\end{theorem}

\begin{proof}
Let $a_1,b_1,c_1 \in \mathcal{L}_1$, $a_2,b_2,c_2 \in \mathcal{L}_2$. It is known that $(\mathcal L,\circ_\lambda)$ is a commutative associative conformal superalgebra.
Then we show that $(\mathcal L,\ast_\lambda)$ is a {left-symmetric conformal superalgebra}.
Set 
$\varepsilon_0=(\abs {a_1}+\abs {a_2})(\abs {b_1}+\abs {b_2}),\varepsilon_1=(-1)^{|a_2||b_1|+(|a_2|+|b_2|)|c_1|},\varepsilon_2=(-1)^{\varepsilon_0+|b_2||a_1|+(|a_2|+|b_2|)|c_1|}.$
we compute
\begin{align*}
        &((a_1\otimes a_2)\ast_\lambda (b_1\otimes b_2))\ast_{\lambda+\mu}(c_1\otimes c_2)-(a_1\otimes a_2)\ast_\lambda((b_1\otimes b_2)\ast_\mu (c_1\otimes c_2))\\
        &-(-1)^{\varepsilon_0}\big(((b_1\otimes b_2)\ast_\mu (a_1\otimes a_2))\ast_{\lambda+\mu}(c_1\otimes c_2)-(b_1\otimes b_2)\ast_\mu((a_1\otimes a_2)\ast_\lambda (c_1\otimes c_2))\big)\\
       =\ &\varepsilon_1((a_1\ast^{1}_\lambda b_1)\ast^{1}_{\lambda+\mu} c_1)\otimes ((a_2\circ^{2}_\lambda b_2)\circ^{2}_{\lambda+\mu} c_2)
       -\varepsilon_1(a_1\ast^{1}_\lambda (b_1\ast^{1}_\mu c_1))\otimes (a_2\circ^{2}_\lambda (b_2\circ^{2}_\mu c_2))\\
         -\ &\varepsilon_2((b_1\ast^{1}_\mu a_1)\ast^{1}_{\lambda+\mu} c_1)\otimes ((b_2\circ^{2}_\mu a_2)\circ^{2}_{\lambda+\mu} c_2)
        +\varepsilon_2(b_1\ast^{1}_\mu (a_1\ast^{1}_\lambda c_1))\otimes (b_2\circ^{2}_\mu (a_2\circ^{2}_\lambda c_2))\\
       +\  &\varepsilon_1((a_1\ast^{1}_\lambda b_1)\circ^{1}_{\lambda+\mu} c_1)\otimes ((a_2\circ^{2}_\lambda b_2)\ast^{2}_{\lambda+\mu} c_2)
      -\varepsilon_1(a_1\ast^{1}_\lambda(b_1\circ^{1}_\mu c_1))\otimes (a_2\circ^{2}_\lambda (b_2\ast^{2}_\mu c_2))\\
      -\ &\varepsilon_2((b_1\ast^{1}_\mu a_1)\circ^{1}_{\lambda+\mu} c_1)\otimes ((b_2\circ^{2}_\mu a_2)\ast^{2}_{\lambda+\mu} c_2)
      +\varepsilon_2(b_1\ast^{1}_\mu (a_1\circ^{1}_\lambda c_1))\otimes (b_2\circ^{2}_\mu (a_2\ast^{2}_\lambda c_2))\\
       +\ &\varepsilon_1((a_1\circ^{1}_\lambda b_1)\ast^{1}_{\lambda+\mu} c_1)\otimes ((a_2\ast^{2}_\lambda b_2)\circ^{2}_{\lambda+\mu} c_2)
         -\varepsilon_1(a_1\circ^{1}_\lambda (b_1\ast^{1}_\mu c_1))\otimes (a_2\ast^{2}_\lambda (b_2\circ^{2}_\mu c_2))\\
         -\ &\varepsilon_2((b_1\circ^{1}_\mu a_1)\ast^{1}_{\lambda+\mu} c_1)\otimes ((b_2\ast^{2}_\mu a_2)\circ^{2}_{\lambda+\mu} c_2)
      +\varepsilon_2(b_1\circ^{1}_\mu (a_1\ast^{1}_\lambda c_1))\otimes (b_2\ast^{2}_\mu (a_2\circ^{2}_\lambda c_2))\\
       +\ &\varepsilon_1((a_1\circ^{1}_\lambda b_1)\circ^{1}_{\lambda+\mu} c_1)\otimes ((a_2\ast^{2}_\lambda b_2)\ast^{2}_{\lambda+\mu} c_2)
       -\varepsilon_1(a_1\circ^{1}_\lambda (b_1\circ^{1}_\mu c_1))\otimes (a_2\ast^{2}_\lambda (b_2\ast^{2}_\mu c_2))\\
        -\ &\varepsilon_2((b_1\circ^{1}_\mu a_1)\circ^{1}_{\lambda+\mu} c_1)\otimes ((b_2\ast^{2}_\mu a_2)\ast^{2}_{\lambda+\mu} c_2)
        +\varepsilon_2(b_1\circ^{1}_\mu (a_1\circ^{1}_\lambda c_1))\otimes (b_2\ast^{2}_\mu (a_2\ast^{2}_\lambda c_2)).
   \end{align*}

\noindent Observe that $\varepsilon_2=(-1)^{\abs {a_1}\abs{b_1}+\abs {a_2}\abs{b_2}}\varepsilon_1$ and
   \begin{align*}
     &\left(a\circ_{\lambda} b\right)\circ_{\lambda+\mu} c= a\circ_{\lambda}\left(b\circ_{\mu} c\right)=(-1)^{\abs {a}\abs{b}}b\circ_\mu\left(a\circ_\lambda c\right)=(-1)^{\abs {a}\abs{b}}\left(b\circ_\mu a\right)\circ_{\lambda+\mu} c,\\
     &\left(a\circ_{\lambda} b\right)\ast_{\lambda+\mu} c= a\circ_{\lambda}\left(b\ast_{\mu} c\right)=(-1)^{\abs {a}\abs{b}}b\circ_\mu\left(a\ast_\lambda c\right)=(-1)^{\abs {a}\abs{b}}\left(b\circ_\mu a\right)\ast_{\lambda+\mu} c.
   \end{align*}
Hence, we set $l_1=(-1)^{\abs {a_1}\abs{b_1}},l_2=(-1)^{\abs {a_2}\abs{b_2}}$ and obtain
\begin{align*}
        &((a_1\otimes a_2)\ast_\lambda (b_1\otimes b_2))\ast_{\lambda+\mu}(c_1\otimes c_2)-(a_1\otimes a_2)\ast_\lambda((b_1\otimes b_2)\ast_\mu (c_1\otimes c_2))\\
        &-(-1)^{\varepsilon_0}\big(((b_1\otimes b_2)\ast_\mu (a_1\otimes a_2))\ast_{\lambda+\mu}(c_1\otimes c_2)-(b_1\otimes b_2)\ast_\mu((a_1\otimes a_2)\ast_\lambda (c_1\otimes c_2))\big)\\
      =\ &\varepsilon_1\big((a_1\ast^{1}_\lambda b_1)\ast^{1}_{\lambda+\mu} c_1-a_1\ast^{1}_\lambda (b_1\ast^{1}_\mu c_1)-l_1(b_1\ast^{1}_\mu a_1)\ast^{1}_{\lambda+\mu} c_1+l_1b_1\ast^{1}_\mu (a_1\ast^{1}_\lambda c_1)\big)\\
      &\otimes ((a_2\circ^{2}_\lambda b_2)\circ^{2}_{\lambda+\mu} c_2)\ +\  
      \varepsilon_1((a_1\circ^{1}_\lambda b_1)\ast^{1}_{\lambda+\mu} c_1)\ \otimes \\
      &\big((a_2\ast^{2}_\lambda b_2)\circ^{2}_{\lambda+\mu} c_2-a_2\ast^{2}_\lambda (b_2\circ^{2}_\mu c_2)-l_2(b_2\ast^{2}_\mu a_2)\circ^{2}_{\lambda+\mu} c_2+l_2b_2\ast^{2}_\mu (a_2\circ^{2}_\lambda c_2)\big)\\
       +\ &\varepsilon_1\big((a_1\ast^{1}_\lambda b_1)\circ^{1}_{\lambda+\mu} c_1-a_1\ast^{1}_\lambda(b_1\circ^{1}_\mu c_1)-l_1(b_1\ast^{1}_\mu a_1)\circ^{1}_{\lambda+\mu} c_1+l_1b_1\ast^{1}_\mu (a_1\circ^{1}_\lambda c_1)\big)\\
       &\otimes ((a_2\circ^{2}_\lambda b_2)\ast^{2}_{\lambda+\mu} c_2)\ +\
         \varepsilon_1((a_1\circ^{1}_\lambda b_1)\circ^{1}_{\lambda+\mu} c_1)\ \otimes \\
         &\big((a_2\ast^{2}_\lambda b_2)\ast^{2}_{\lambda+\mu} c_2-a_2\ast^{2}_\lambda (b_2\ast^{2}_\mu c_2)-l_2(b_2\ast^{2}_\mu a_2)\ast^{2}_{\lambda+\mu} c_2+l_2b_2\ast^{2}_\mu (a_2\ast^{2}_\lambda c_2)\big)\\
\overset{\textbf{\eqref{eq:4-8}}}{=}  & \varepsilon_1((a_1\ast^{1}_\lambda b_1)\ast^{1}_{\lambda+\mu} c_1-a_1\ast^{1}_\lambda (b_1\ast^{1}_\mu c_1)-l_1(b_1\ast^{1}_\mu a_1)\ast^{1}_{\lambda+\mu} c_1+l_1b_1\ast^{1}_\mu (a_1\ast^{1}_\lambda c_1))\\
      &\otimes ((a_2\circ^{2}_\lambda b_2)\circ^{2}_{\lambda+\mu} c_2)\ +\  
         \varepsilon_1((a_1\circ^{1}_\lambda b_1)\circ^{1}_{\lambda+\mu} c_1)\ \otimes \\
         &((a_2\ast^{2}_\lambda b_2)\ast^{2}_{\lambda+\mu} c_2-a_2\ast^{2}_\lambda (b_2\ast^{2}_\mu c_2)-l_2(b_2\ast^{2}_\mu a_2)\ast^{2}_{\lambda+\mu} c_2+l_2b_2\ast^{2}_\mu (a_2\ast^{2}_\lambda c_2))\\
\overset{\textbf{\eqref{eq:2-18}}}{ =} &\ 0.
\end{align*}
Therefore, $(\mathcal L,\ast_\lambda)$ is a {left-symmetric conformal superalgebra}. Moreover, we have
\begin{align*}
        &((a_1\otimes a_2)\ast_\lambda (b_1\otimes b_2))\circ_{\lambda+\mu}(c_1\otimes c_2)-(a_1\otimes a_2)\ast_\lambda((b_1\otimes b_2)\circ_\mu (c_1\otimes c_2))\\
        &-(-1)^{\varepsilon_0}\big(((b_1\otimes b_2)\ast_\mu (a_1\otimes a_2))\circ_{\lambda+\mu}(c_1\otimes c_2)-(b_1\otimes b_2)\ast_\mu((a_1\otimes a_2)\circ_\lambda (c_1\otimes c_2))\big)\\
      =\ &\varepsilon_1((a_1\ast^{1}_\lambda b_1)\circ^{1}_{\lambda+\mu} c_1)\otimes ((a_2\circ^{2}_\lambda b_2)\circ^{2}_{\lambda+\mu} c_2)
         -\varepsilon_1(a_1\ast^{1}_\lambda (b_1\circ^{1}_\mu c_1))\otimes (a_2\circ^{2}_\lambda (b_2\circ^{2}_\mu c_2))\\
         -\ &\varepsilon_2((b_1\ast^{1}_\mu a_1)\circ^{1}_{\lambda+\mu} c_1)\otimes ((b_2\circ^{2}_\mu a_2)\circ^{2}_{\lambda+\mu} c_2)
         +\varepsilon_2(b_1\ast^{1}_\mu (a_1\circ^{1}_\lambda c_1))\otimes (b_2\circ^{2}_\mu (a_2\circ^{2}_\lambda c_2))\\
        +\ &\varepsilon_1((a_1\circ^{1}_\lambda b_1)\circ^{1}_{\lambda+\mu} c_1)\otimes ((a_2\ast^{2}_\lambda b_2)\circ^{2}_{\lambda+\mu} c_2)
        -\varepsilon_1(a_1\circ^{1}_\lambda (b_1\circ^{1}_\mu c_1))\otimes (a_2\ast^{2}_\lambda (b_2\circ^{2}_\mu c_2))\\
        -\ &\varepsilon_2((b_1\circ^{1}_\mu a_1)\circ^{1}_{\lambda+\mu} c_1)\otimes ((b_2\ast^{2}_\mu a_2)\circ^{2}_{\lambda+\mu} c_2)
        +\varepsilon_2(b_1\circ^{1}_\mu (a_1\circ^{1}_\lambda c_1))\otimes (b_2\ast^{2}_\mu (a_2\circ^{2}_\lambda c_2))\\
        =\ &\varepsilon_1\big((a_1\ast^{1}_\lambda b_1)\circ^{1}_{\lambda+\mu} c_1-a_1\ast^{1}_\lambda (b_1\circ^{1}_\mu c_1)-l_1(b_1\ast^{1}_\mu a_1)\circ^{1}_{\lambda+\mu} c_1+l_1b_1\ast^{1}_\mu (a_1\circ^{1}_\lambda c_1)\big)\\
        &\otimes ((a_2\circ^{2}_\lambda b_2)\circ^{2}_{\lambda+\mu} c_2)\ +\ 
        \varepsilon_1((a_1\circ^{1}_\lambda b_1)\circ^{1}_{\lambda+\mu} c_1)\ \otimes \\
        &\big((a_2\ast^{2}_\lambda b_2)\circ^{2}_{\lambda+\mu} c_2-a_2\ast^{2}_\lambda (b_2\circ^{2}_\mu c_2)-l_2(b_2\ast^{2}_\mu a_2)\circ^{2}_{\lambda+\mu} c_2+l_2b_2\ast^{2}_\mu (a_2\circ^{2}_\lambda c_2)\big)\\
\overset{\textbf{\eqref{eq:4-8}}}{=} & \ 0.
\end{align*}
It remains to check the compatibility condition {\eqref{eq:4-7}}. By a direct computation, we get
   \begin{align*}
        &((a_1\otimes a_2)\circ_\lambda (b_1\otimes b_2))\ast_{\lambda+\mu}(c_1\otimes c_2)
         =(-1)^{\abs {a_2}\abs {b_1}}((a_1\circ^{1}_\lambda b_1)\otimes (a_2\circ^{2}_\lambda b_2))\ast_{\lambda+\mu}(c_1\otimes c_2)\\
         =&\varepsilon_1\big(((a_1\circ^{1}_\lambda b_1)\ast^{1}_{\lambda+\mu} c_1)\otimes ((a_2\circ^{2}_\lambda b_2)\circ^{2}_{\lambda+\mu} c_2)+((a_1\circ^{1}_\lambda b_1)\circ^{1}_{\lambda+\mu} c_1)\otimes ((a_2\circ^{2}_\lambda b_2)\ast^{2}_{\lambda+\mu} c_2)\big)\\
        =&\varepsilon_1\big((a_1\circ^{1}_\lambda( b_1\ast^{1}_{\mu} c_1))\otimes (a_2\circ^{2}_\lambda (b_2\circ^{2}_{\mu} c_2))+(a_1\circ^{1}_\lambda (b_1\circ^{1}_{\mu} c_1))\otimes (a_2\circ^{2}_\lambda (b_2\ast^{2}_{\mu} c_2))\big)\\
       =&(-1)^{\abs {b_2}\abs {c_1}}\big((a_1\otimes a_2)\circ_\lambda(( b_1\ast^{1}_{\mu} c_1)\otimes  (b_2\circ^{2}_{\mu} c_2))
       +(a_1\otimes a_2)\circ_\lambda(( b_1\circ^{1}_{\mu} c_1)\otimes  (b_2\ast^{2}_{\mu} c_2))\big)\\
         =&(-1)^{\abs {b_2}\abs {c_1}}(a_1\otimes a_2)\circ_\lambda(( b_1\ast^{1}_{\mu} c_1)\otimes  (b_2\circ^{2}_{\mu} c_2)
       + (b_1\circ^{1}_{\mu} c_1)\otimes  (b_2\ast^{2}_{\mu} c_2))\\
       =&(a_1\otimes a_2)\circ_\lambda((b_1\otimes b_2)\ast_{\mu}(c_1\otimes c_2)).
      \end{align*}
      
Altogether, $(\mathcal L,\circ_\lambda,\ast_\lambda)$ is a pre-Lie Poisson conformal superalgebra.
\end{proof}

\section{Compatible TPCSA structures on Lie conformal superalgebras of rank (1+1)}\label{sec5}

In this section, we classify all compatible  TPCSA structures on Lie conformal superalgebras of rank (1+1).
Let $R=R_{\bar 0}\oplus R_{\bar 1}$ be a free Lie conformal superalgebra of rank (1+1), where $ R_{\bar 0}=\mathbb{C}[\partial]x, R_{\bar 1}=\mathbb{C}[\partial]y$.
We shall use the classification of Lie conformal superalgebras of rank (1+1) given in \cite[Proposition 2.4]{ZCYjmh17}. Up to isomorphism, such a Lie conformal superalgebra is one of the following five types:
\begin{align*}
&R_1: ~ [x_\lambda x]=0,\quad [x_\lambda y]=0,\quad [y_\lambda y]=p(\pa)x,\quad \forall~  p(\pa)\in\C[\pa];\\
&R_2: ~ [x_\lambda x]=0,\quad [y_\lambda y]=0,\quad [x_\lambda y]=q(\lambda)y,\quad \forall~q(\lambda)\in\C[\lambda];\\
&R_3:~ [x_\lambda x]=(\pa+2\lambda)x,\quad [x_\lambda y]=0,\quad [y_\lambda y]=0;\\
&R_4: ~ [x_\lambda x]=(\pa+2\lambda)x,\quad [x_\lambda y]=(\pa+\beta\lambda+\gamma)y,\quad [y_\lambda y]=0,\quad \forall~\beta,\gamma\in\C;\\
&R_5: ~ [x_\lambda x]=(\pa+2\lambda)x,\quad
[x_\lambda y]=(\pa+\tfrac32\lambda)y,\quad
[y_\lambda y]=\alpha x, \quad \forall~\alpha\in\C.
\end{align*}

We begin  with the following definition.
\begin{definition}
Let $(A,[{\cdot_ \lambda \cdot}])$ be a {Lie conformal superalgebra}. A compatible TPCSA structure over $A$ is an even $\lambda$-product $\circ_{\lambda}$ on $A$, such that $(A,\circ_{\lambda},[{\cdot_ \lambda \cdot}])$ forms a TPCSA.
\end{definition}

Obviously, the $\lambda$-product $a\circ_\lambda b=0$ defines a compatible TPCSA structure over the Lie conformal superalgebra $A$. This is referred to as the {\bf trivial} compatible TPCSA structure over $A$. Consequently, our primary task is to determine the existence of nontrivial compatible TPCSA structures on
$A$, and to classify such structures if they exist. 

Let $\circ_\lambda$ be an even $\lambda$-product on $R$. By the definition of a TPCSA, we can assume that 
\begin{equation}\label{eq:general-product}
 x\circ_\lambda x=f(\pa,\lambda)x,
 \quad
 x\circ_\lambda y=g(\pa,\lambda)y,
 \quad
 y\circ_\lambda y=h(\pa,\lambda)x,
\end{equation}
where $f(\pa,\lambda),g(\pa,\lambda),h(\pa,\lambda)\in\C[\pa,\lambda]$. By commutativity, we get
\begin{equation}\label{eq:supercommutativity}
 f(\pa,\lambda)=f(\pa,-\pa-\lambda),
 \quad
 y\circ_\lambda x=g(\pa,-\pa-\lambda)y,
 \quad
 h(\pa,\lambda)=-h(\pa,-\pa-\lambda).
\end{equation}
By associativity, we have $x\circ_\lambda(x\circ_\mu x)=(x\circ_\lambda x)\circ_{\lambda+\mu}x$, this is equivalent to
    \begin{align}\label{ly3}
        f(\partial+\lambda, \mu) f(\partial, \lambda)=f(-\lambda-\mu, \lambda) f(\partial, \lambda+\mu),
    \end{align}
which implies $deg_{\partial} f(\partial, \lambda)+deg_{\partial} f(\partial, \lambda)=deg_{\partial} f(\partial, \lambda)$,
where we use the notation $deg_{\partial} f(\partial, \lambda)$ to denote the highest degree of $\partial$ in $f(\partial, \lambda)$. Hence, $deg_{\partial} f(\partial, \lambda)=0$.
Now we can suppose $f(\partial, \lambda)=f(\lambda)$ for some $f(\lambda)\in \mathbb{C}[\lambda]$. Then Eq.~\eqref{ly3} reduces to 
    \begin{align}
        f(\mu) f(\lambda)=f(\lambda) f(\lambda+\mu),
    \end{align}
which implies $f(\lambda)=c$ for some $c\in \mathbb{C}$. 

Moreover, we have $x\circ_\lambda(x\circ_\mu y)=(x\circ_\lambda x)\circ_{\lambda+\mu}y$, hence
    \begin{align}
        g(\partial+\lambda, \mu) g(\partial, \lambda)=c g(\partial, \lambda+\mu).
    \end{align}

(1) If $c=0$, then $g(\partial, \lambda)=0$. 
(2) If $c\neq0$, considering the highest degree of $\partial$ on both sides,  
we get $deg_{\partial} g(\partial, \lambda)=0$, leading to $g(\partial, \lambda)=g(\lambda)$ for some $g(\lambda)\in\mathbb{C}[\lambda]$.
It follows that $g(\mu) g(\lambda)=c g(\lambda+\mu)$, which implies $g(\lambda)=0$ or $g(\lambda)=c$.

The above elementary computations suggest that the possible forms of
$f(\partial,\lambda)$ and $g(\partial,\lambda)$ are restricted to the following two cases:
    \begin{align}
      &(i)~ f(\partial, \lambda)=c_1,\quad g(\partial, \lambda)=0,\quad c_1\in\C\backslash\{0\}; \label{type1}\\
      &(ii)~ f(\partial, \lambda)=c_2,\quad g(\partial, \lambda)=c_2,\quad c_2\in\C. \label{type2}
    \end{align}

Next we proceed to analyze the five types $R_1,\dots,R_5$ separately.

{\bf C{ase} 1:} {Type $R_1$}.
Assume first that $p(\pa)\neq0$, taking $(a,b,c)=(y,y,y)$ in Eq.~\eqref{rule3}, we obtain
    \begin{align*}
        2y\circ_\lambda [y_\mu y]= [(y\circ_\lambda y)_{\lambda+\mu}y]-[y_\mu(y\circ_\lambda y)].
    \end{align*}
Since $y\circ_\lambda y\in \mathbb C[\partial]x$ and all mixed
brackets between $x$ and $y$ vanish, the right-hand side is zero. Hence
$$
 2p(\pa+\lambda)g(\pa,-\pa-\lambda)y=0.
$$
As $\C[\pa,\lambda]$ is an integral domain and $p(\pa+\lambda)\neq0$, we get $g(\pa,-\pa-\lambda)=0$, i.e., $g(\pa,\lambda)=0$.

Next, taking $(a,b,c)=(x,y,y)$ in Eq.~\eqref{rule3} gives
\[
 2x\circ_\lambda [y_\mu y]
 =
 [(x\circ_\lambda y)_{\lambda+\mu}y]
 +[y_\mu(x\circ_\lambda y)].
\]
Since $g=0$, the right-hand side vanishes, and therefore $2p(\pa+\lambda)f(\pa,\lambda)x=0$.
It follows that  $ f(\pa,\lambda)=0$.

Consequently, the only possible nonzero $\lambda$-product is $ y\circ_\lambda y=h(\pa,\lambda)x$.
By commutativity, $h$ must satisfy $h(\pa,\lambda)=-h(\pa,-\pa-\lambda)$.
Equivalently, writing $T=\partial+2\lambda$, the polynomial $h$ is odd with respect to $T$.
Hence there exists a polynomial $\Phi(s,t)\in\mathbb C[s,t]$ such that
\[
h(\partial,\lambda)
=
(\partial+2\lambda)\Phi\bigl(\partial,(\partial+2\lambda)^2\bigr).
\]
Conversely, one can verify that any such polynomial $h$ defines an even commutative associative $\lambda$-product on $R_1$.

If $p(\pa)=0$, then the Lie conformal superalgebra $R_1$ is abelian, and every even commutative associative $\lambda$-product gives a TPCSA structure.


{\bf C{ase} 2:} {Type $R_2$}. If $q(\lambda)=0$, then the Lie conformal superalgebra $R_2$ is abelian, and the argument is the same as in the case $p(\partial)=0$ above.

Assume that $q(\lambda)\neq0$, taking $(a,b,c)=(y,x,y)$ in Eq.~\eqref{rule3}, we obtain
\[
 2y\circ_\lambda [x_\mu y]
 =
 [(y\circ_\lambda x)_{\lambda+\mu}y]
 +[x_\mu(y\circ_\lambda y)].
\]
Since $[y_\lambda y]=0$ and $[x_\lambda x]=0$, we get $2q(\mu)h(\pa,\lambda)x=0$.
Hence $ h(\pa,\lambda)=0$.

Taking $(a,b,c)=(x,x,y)$ in Eq.~\eqref{rule3}, we get
\[
 2x\circ_\lambda [x_\mu y]
 =
 [(x\circ_\lambda x)_{\lambda+\mu}y]
 +[x_\mu(x\circ_\lambda y)].
\]
This is equivalent to
\begin{equation*}
 q(\mu)\bigl(2g(\pa,\lambda)-g(\pa+\mu,\lambda)\bigr)
 =f(-\lambda-\mu,\lambda)q(\lambda+\mu).
\end{equation*}
Since $q(\lambda)\neq0$, the pair $(f,g)$ cannot be of the form \eqref{type1}. Indeed,
if $g(\partial,\lambda)=0$ and $f(\partial,\lambda)=c_1\neq0$, then the above identity gives
$0=c_1q(\lambda+\mu)$, a contradiction. Thus $(f,g)$ must be of the form \eqref{type2}.
Therefore, $c_2(q(\mu)-q(\lambda+\mu))=0$.

(1) If $q(\lambda)$ is a nonzero constant, then $c_2\in \C$.
(2) If $q(\lambda)$ is nonconstant, then $c_2=0$.

{\bf C{ase} 3:} {Type $R_3$}.
Taking $(a,b,c)=(y,x,x)$ in Eq.~\eqref{rule3}, we get
\[
 2y\circ_\lambda [x_\mu x]
 =
 [(y\circ_\lambda x)_{\lambda+\mu}x]
 +[x_\mu(y\circ_\lambda x)].
\]
Since all brackets involving both $x$ and $y$ are zero, the right-hand side is zero. Hence
\[
 2(\pa+\lambda+2\mu)g(\pa,-\pa-\lambda)y=0,
\]
and therefore $g(\pa,-\pa-\lambda)=0$, i.e., $g(\pa,\lambda)=0$.

Similarly, taking $(a,b,c)=(y,x,y)$ in Eq.~\eqref{rule3} gives $ [x_\mu(y\circ_\lambda y)]=0$, that is
\[
 (\pa+2\mu)h(\pa+\mu,\lambda)x=0.
\]
Hence $h(\pa+\mu,\lambda)=0$, i.e., $h(\pa,\lambda)=0$.

As previously mentioned, $f(\pa,\lambda)=c$ for some $c\in \mathbb{C}$. 
Thus the compatible TPCSA structures on $R_3$ are precisely
\begin{equation*}
 x\circ_\lambda x=cx,
 \quad
 x\circ_\lambda y=0,
 \quad
 y\circ_\lambda y=0,
 \quad
 c\in \C.
\end{equation*}

{\bf C{ase} 4:} {Type $R_4$}.
Taking $(a,b,c)=(x,x,y)$ in Eq.~\eqref{rule3}, we obtain
\[
 2x\circ_\lambda [x_\mu y]
 =
 [(x\circ_\lambda x)_{\lambda+\mu}y]
 +[x_\mu(x\circ_\lambda y)].
\]
Since $f(\pa,\lambda)=c$ for some $c\in \mathbb{C}$, we have
\begin{equation}\label{eq:R4}
 2(\pa+\lambda+\beta\mu+\gamma)g(\pa,\lambda)
 =
 c(\pa+\beta\lambda+\beta\mu+\gamma)
 +(\pa+\beta\mu+\gamma)g(\pa+\mu,\lambda).
\end{equation}
Setting $\mu=0$ in Eq.~\eqref{eq:R4}, we get
\begin{equation}\label{eq:R4-zero}
 (\pa+2\lambda+\gamma)g(\pa,\lambda)=c(\pa+\beta\lambda+\gamma).
\end{equation}
If $\beta=2$, then $(\partial+2\lambda+\gamma)(g(\pa,\lambda)-c)=0$, hence $g(\pa,\lambda)=c$.
If $\beta\neq2$, by comparing the similar terms in Eq.~\eqref{eq:R4-zero}, one can obtain $g(\partial, \lambda)=c=0$.

It remains to determine polynomial $h$. Taking $(a,b,c)=(y,x,y)$ in Eq.~\eqref{rule3} gives
\[
 2y\circ_\lambda [x_\mu y]
 =
 [(y\circ_\lambda x)_{\lambda+\mu}y]
 +[x_\mu(y\circ_\lambda y)].
\]
Since $[y_\lambda y]=0$, we obtain
\begin{equation}\label{eq:R4-H-equation}
 2(\pa+\lambda+\beta\mu+\gamma)h(\pa,\lambda)
 =
 (\pa+2\mu)h(\pa+\mu,\lambda).
\end{equation}
Setting $\mu=0$ in \eqref{eq:R4-H-equation}, we get $(\pa+2\lambda+2\gamma)h(\pa,\lambda)=0$, leading to $h(\pa,\lambda)=0$.

Consequently, if $\beta=2$, then
\[
 x\circ_\lambda x=cx,
 \quad
 x\circ_\lambda y=cy,
 \quad
 y\circ_\lambda y=0,
 \quad
 c\in \C,
\]
whereas if $\beta\neq2$, then the $\lambda$-product is trivial.

{\bf C{ase} 5:} {Type $R_5$}.
Taking $(a,b,c)=(x,x,y)$ in Eq.~\eqref{rule3}, we obtain
\[
 2x\circ_\lambda [x_\mu y]
 =
 [(x\circ_\lambda x)_{\lambda+\mu}y]
 +[x_\mu(x\circ_\lambda y)].
\]
Since $f(\pa,\lambda)=c$ for some $c\in \mathbb{C}$, we have
\begin{equation}\label{eq:R4s}
 2(\pa+\lambda+\tfrac32\mu)g(\pa,\lambda)
 =
 c(\pa+\tfrac32\lambda+\tfrac32\mu)
 +(\pa+\tfrac32\mu)g(\pa+\mu,\lambda).
\end{equation}
Setting $\mu=0$ in Eq.~\eqref{eq:R4s}, we get
$$ (\pa+2\lambda)g(\pa,\lambda)=c(\pa+\tfrac32\lambda).$$
Since $\pa+\tfrac32\lambda$ is not divisible by $\pa+2\lambda$, we must have $c=0$.
Thus $f(\pa,\lambda)=g(\pa,\lambda)=0$.

Moreover, taking $(a,b,c)=(y,x,y)$ in Eq.~\eqref{rule3} gives
\[
 2y\circ_\lambda [x_\mu y]
 =
 [(y\circ_\lambda x)_{\lambda+\mu}y]
 +[x_\mu(y\circ_\lambda y)].
\]
Since $y\circ_\lambda x=0$, we obtain
\begin{equation}\label{eq:R4syy}
 2(\pa+\lambda+\tfrac32\mu)h(\pa,\lambda)
 =
 (\pa+2\mu)h(\pa+\mu,\lambda).
\end{equation}
Setting $\mu=0$ in Eq.~\eqref{eq:R4syy}, we get $(\pa+2\lambda)h(\pa,\lambda)=0$, hence $h(\pa,\lambda)=0$.

Therefore, the Lie conformal superalgebra $R_5$ admits only the trivial TPCSA structure.

The preceding case-by-case computations give all necessary forms of the
$\lambda$-product, we summarize the discussions above in the following theorem:

\begin{theorem}\label{thm:rank-one-one-classification}
Let $R$ be one of the Lie conformal superalgebras $R_1,\ldots,R_5$ listed
above, and let $\circ_\lambda$ be an even $\lambda$-product on $R$. Then
$\circ_\lambda$ defines a compatible TPCSA
structure on $R$ if and only if it is one of the following: 
\begin{enumerate}[\upshape(1)]
\item For $R_1$: If $p(\pa)\neq0$, then
\[
 x\circ_\lambda x=0,
 \quad
 x\circ_\lambda y=0,
\]
and
\[
y\circ_\lambda y
=
(\partial+2\lambda)\Phi\bigl(\partial,(\partial+2\lambda)^2\bigr)x,
\quad
\Phi(s,t)\in\mathbb C[s,t].
\]
(Each polynomial $\Phi(s,t)\in\mathbb C[s,t]$ defines such a structure.)

If $p(\pa)=0$, then $R_1$ is abelian, and every even commutative
associative $\lambda$-product gives such a structure.

\item For $R_2$: If $q(\lambda)=0$, then $R_2$ is abelian, and every even
commutative associative $\lambda$-product gives such a structure. 

If $q(\lambda)$ is a nonzero constant, then
\[
 x\circ_\lambda x=cx,
 \quad
 x\circ_\lambda y=cy,
 \quad
 y\circ_\lambda y=0,
 \quad
  c\in\C.
\]

If $q(\lambda)$ is nonconstant, then only the trivial structure exists:
\[
 x\circ_\lambda x=x\circ_\lambda y=y\circ_\lambda y=0.
\]

\item For $R_3$: Every structure is given by
\[
 x\circ_\lambda x=cx,
 \quad
 x\circ_\lambda y=0,
 \quad
 y\circ_\lambda y=0,
  \quad
  c\in\C.
\]

\item For $R_4$: If $\beta=2$, then
\[
 x\circ_\lambda x=cx,
 \quad
 x\circ_\lambda y=cy,
 \quad
 y\circ_\lambda y=0,
  \quad
  c\in\C.
\]

If $\beta\neq2$, then only the trivial structure exists:
\[
 x\circ_\lambda x=x\circ_\lambda y=y\circ_\lambda y=0.
\]

\item For $R_5$: Only the trivial structure exists:
\[
 x\circ_\lambda x=x\circ_\lambda y=y\circ_\lambda y=0.
\]
\end{enumerate}
\end{theorem}

\begin{proof}

The necessity is precisely the content of the five cases considered above.
Conversely, for each $\lambda$-product listed in the theorem, evenness and
commutativity are immediate from the construction, and associativity
follows directly from the displayed formulas. A direct substitution of these
$\lambda$-products into Eq.~\eqref{rule3} verifies the transposed super-Leibniz rule
in each case. Hence the listed necessary forms are also sufficient.

\end{proof}




\bibliography{tpcsa}

@article {DKjjm13,
    AUTHOR = {De Sole, Alberto and Kac, Victor G.},
     TITLE = {The variational {P}oisson cohomology},
   JOURNAL = {Jpn. J. Math.},
  FJOURNAL = {Japanese Journal of Mathematics},
    VOLUME = {8},
      YEAR = {2013},
    NUMBER = {1},
     PAGES = {1--145},
      ISSN = {0289-2316,1861-3624},
   MRCLASS = {17B80 (17B56 17B69 37K10 37K30)},
  MRNUMBER = {3067293},
MRREVIEWER = {Vladimir\ Dotsenko},
       DOI = {10.1007/s11537-013-1124-3},
       URL = {https://doi.org/10.1007/s11537-013-1124-3},
}

@incollection {FKRb04,
    AUTHOR = {Fattori, Davide and Kac, Victor G. and Retakh, Alexander},
     TITLE = {Structure theory of finite {L}ie conformal superalgebras},
 BOOKTITLE = {Lie theory and its applications in physics {V}},
     PAGES = {27--63},
 PUBLISHER = {World Sci. Publ., River Edge, NJ},
      YEAR = {2004},
      ISBN = {981-238-936-9},
   MRCLASS = {17B68 (81R10 81T40)},
  MRNUMBER = {2172172},
MRREVIEWER = {Stanislav\ Z.\ Pakuliak},
       DOI = {10.1142/9789812702562\_0002},
}

@article {KKPjne21,
    AUTHOR = {Kolesnikov, Pavel S. and Kozlov, Roman A. and Panasenko,
              Aleksander S.},
     TITLE = {Quadratic {L}ie conformal superalgebras related to {N}ovikov
              superalgebras},
   JOURNAL = {J. Noncommut. Geom.},
  FJOURNAL = {Journal of Noncommutative Geometry},
    VOLUME = {15},
      YEAR = {2021},
    NUMBER = {4},
     PAGES = {1485--1500},
      ISSN = {1661-6952,1661-6960},
   MRCLASS = {17B69 (17B63 37K30)},
  MRNUMBER = {4357073},
MRREVIEWER = {Youjun\ Tan},
       DOI = {10.4171/jncg/445},
}

@book {Kac98,
    AUTHOR = {Kac, Victor G.},
     TITLE = {Vertex algebras for beginners},
    SERIES = {University Lecture Series},
    VOLUME = {10},
   EDITION = {Second},
 PUBLISHER = {American Mathematical Society, Providence, RI},
      YEAR = {1998},
     PAGES = {vi+201},
      ISBN = {0-8218-1396-X},
   MRCLASS = {17B69},
  MRNUMBER = {1651389},
       DOI = {10.1090/ulect/010},
       URL = {https://doi.org/10.1090/ulect/010},
}

@article {CSKajm97,
    AUTHOR = {Cheng, Shun-Jen and Kac, Victor G.},
     TITLE = {Conformal modules},
   JOURNAL = {Asian J. Math.},
  FJOURNAL = {Asian Journal of Mathematics},
    VOLUME = {1},
      YEAR = {1997},
    NUMBER = {1},
     PAGES = {181--193},
      ISSN = {1093-6106,1945-0036},
   MRCLASS = {17B68 (17B66 81R10)},
  MRNUMBER = {1480993},
MRREVIEWER = {Christoph\ Schweigert},
       DOI = {10.4310/AJM.1997.v1.n1.a6},
       URL = {https://doi.org/10.4310/AJM.1997.v1.n1.a6},
}

@article {YCHase17,
    AUTHOR = {Yuan, Lamei and Chen, Sheng and He, Caixia},
     TITLE = {Hom-{G}el'fand-{D}orfman super-bialgebras and {H}om-{L}ie
              conformal superalgebras},
   JOURNAL = {Acta Math. Sin. (Engl. Ser.)},
  FJOURNAL = {Acta Mathematica Sinica (English Series)},
    VOLUME = {33},
      YEAR = {2017},
    NUMBER = {1},
     PAGES = {96--116},
      ISSN = {1439-8516,1439-7617},
   MRCLASS = {17B69 (17B60 17D25)},
  MRNUMBER = {3581609},
MRREVIEWER = {Chengming\ Bai},
       DOI = {10.1007/s10114-016-6074-2},
}

@article {HL15,
    AUTHOR = {Hong, Yanyong and Li, Fang},
     TITLE = {Left-symmetric conformal algebras and vertex algebras},
   JOURNAL = {J. Pure Appl. Algebra},
  FJOURNAL = {Journal of Pure and Applied Algebra},
    VOLUME = {219},
      YEAR = {2015},
    NUMBER = {8},
     PAGES = {3543--3567},
      ISSN = {0022-4049,1873-1376},
   MRCLASS = {17B69 (17D25)},
  MRNUMBER = {3320236},
MRREVIEWER = {Elisabeth\ Remm},
       DOI = {10.1016/j.jpaa.2014.12.012},
       URL = {https://doi.org/10.1016/j.jpaa.2014.12.012},
}

@article {Hongjlt16,
    AUTHOR = {Hong, Yanyong},
     TITLE = {A class of {L}ie conformal superalgebras in higher dimensions},
   JOURNAL = {J. Lie Theory},
  FJOURNAL = {Journal of Lie Theory},
    VOLUME = {26},
      YEAR = {2016},
    NUMBER = {4},
     PAGES = {1145--1162},
      ISSN = {0949-5932},
   MRCLASS = {17B69 (17B62 17D25)},
  MRNUMBER = {3503807},
}

@article {CThjms24,
    AUTHOR = {Chtioui, Taoufik},
     TITLE = {Hom-{G}el'fand-{D}orfman conformal superbialgebras},
   JOURNAL = {Hacet. J. Math. Stat.},
  FJOURNAL = {Hacettepe Journal of Mathematics and Statistics},
    VOLUME = {53},
      YEAR = {2024},
    NUMBER = {3},
     PAGES = {577--585},
      ISSN = {1303-5010,2651-477X},
   MRCLASS = {17B60 (17B61 17B63)},
  MRNUMBER = {4766521},
       DOI = {10.15672/hujms.1196147},
       URL = {https://doi.org/10.15672/hujms.1196147},
}

@article {FKLrac21,
    AUTHOR = {Ferreira, Bruno Leonardo Macedo and Kaygorodov, Ivan and
              Lopatkin, Viktor},
     TITLE = {{$\frac{1}{2}$}-derivations of {L}ie algebras and transposed
              {P}oisson algebras},
   JOURNAL = {Rev. R. Acad. Cienc. Exactas F\'is. Nat. Ser. A Mat. RACSAM},
  FJOURNAL = {Revista de la Real Academia de Ciencias Exactas, F\'isicas y
              Naturales. Serie A. Matematicas. RACSAM},
    VOLUME = {115},
      YEAR = {2021},
    NUMBER = {3},
     PAGES = {Paper No. 142, 19},
      ISSN = {1578-7303,1579-1505},
   MRCLASS = {17B40 (17A42 17B63)},
  MRNUMBER = {4272639},
MRREVIEWER = {Luiz\ Antonio\ Peresi},
       DOI = {10.1007/s13398-021-01088-2},
}

@unpublished{YF2026,
      title={On transposed {P}oisson conformal algebras}, 
      author={Lamei Yuan and Hao Fang},
      year={2026},
      archivePrefix={arXiv},
      NOTE = {Preprint, 21 pp.},
      url={https://arxiv.org/abs/2603.14735}, 
}

@article {ZCYjmh17,
    AUTHOR = {Zhao, Jun and Chen, Liangyun and Yuan, Lamei},
     TITLE = {Deformations and generalized derivations of {L}ie conformal
              superalgebras},
   JOURNAL = {J. Math. Phys.},
  FJOURNAL = {Journal of Mathematical Physics},
    VOLUME = {58},
      YEAR = {2017},
    NUMBER = {11},
     PAGES = {111702, 17},
      ISSN = {0022-2488,1089-7658},
   MRCLASS = {17B40 (16W25 17B56)},
  MRNUMBER = {3724693},
       DOI = {10.1063/1.5012886},
       URL = {https://doi.org/10.1063/1.5012886},
}

@article {Chen25,
    AUTHOR = {Chen, Fulin and Liao, Xiaoling and Tan, Shaobin and Wang,
              Qing},
     TITLE = {A new construction of {L}ie conformal algebras from formal
              distribution {L}ie algebras},
   JOURNAL = {J. Algebra},
  FJOURNAL = {Journal of Algebra},
    VOLUME = {680},
      YEAR = {2025},
     PAGES = {96--133},
      ISSN = {0021-8693,1090-266X},
   MRCLASS = {17B65 (17B69)},
  MRNUMBER = {4913650},
       DOI = {10.1016/j.jalgebra.2025.04.045},
}

@article {Li96,
    AUTHOR = {Li, Hai-Sheng},
     TITLE = {Local systems of vertex operators, vertex superalgebras and
              modules},
   JOURNAL = {J. Pure Appl. Algebra},
  FJOURNAL = {Journal of Pure and Applied Algebra},
    VOLUME = {109},
      YEAR = {1996},
    NUMBER = {2},
     PAGES = {143--195},
      ISSN = {0022-4049,1873-1376},
   MRCLASS = {17B69},
  MRNUMBER = {1387738},
MRREVIEWER = {Mirko\ Primc},
       DOI = {10.1016/0022-4049(95)00079-8},
       URL = {https://doi.org/10.1016/0022-4049(95)00079-8},
}

@article {Su13,
    AUTHOR = {Su, Yucai and Yuan, Lamei},
     TITLE = {Schr\"odinger-{V}irasoro {L}ie conformal algebra},
   JOURNAL = {J. Math. Phys.},
  FJOURNAL = {Journal of Mathematical Physics},
    VOLUME = {54},
      YEAR = {2013},
    NUMBER = {5},
     PAGES = {053503, 16},
      ISSN = {0022-2488,1089-7658},
   MRCLASS = {17B81 (17B68)},
  MRNUMBER = {3098941},
MRREVIEWER = {Yanyong\ Hong},
       DOI = {10.1063/1.4803029},
}

@incollection {Kac97Locality,
    AUTHOR = {Kac, Victor G.},
     TITLE = {The idea of locality},
 BOOKTITLE = {Physical Applications and Mathematical Aspects of Geometry, Groups and Algebras},
     PAGES = {16--32},
 PUBLISHER = {World Scientific},
   ADDRESS = {Singapore},
      YEAR = {1997},
}

@article {BKV99,
    AUTHOR = {Bakalov, Bojko and Kac, Victor G. and Voronov, Alexander A.},
     TITLE = {Cohomology of conformal algebras},
   JOURNAL = {Comm. Math. Phys.},
  FJOURNAL = {Communications in Mathematical Physics},
    VOLUME = {200},
      YEAR = {1999},
    NUMBER = {3},
     PAGES = {561--598},
      ISSN = {0010-3616,1432-0916},
   MRCLASS = {17B56 (17B66 17B68 81R10)},
  MRNUMBER = {1675121},
MRREVIEWER = {Mirko\ Primc},
       DOI = {10.1007/s002200050541},
       URL = {https://doi.org/10.1007/s002200050541},
}

@incollection {FKja02,
    AUTHOR = {Fattori, Davide and Kac, Victor G.},
     TITLE = {Classification of finite simple {L}ie conformal superalgebras},
      NOTE = {Special issue in celebration of Claudio Procesi's 60th
              birthday},
   JOURNAL = {J. Algebra},
  FJOURNAL = {Journal of Algebra},
    VOLUME = {258},
      YEAR = {2002},
    NUMBER = {1},
     PAGES = {23--59},
      ISSN = {0021-8693,1090-266X},
   MRCLASS = {17B20 (17B68 81R10)},
  MRNUMBER = {1958896},
MRREVIEWER = {Stanislav\ Z.\ Pakuliak},
       DOI = {10.1016/S0021-8693(02)00504-5},
       URL = {https://doi.org/10.1016/S0021-8693(02)00504-5},
}

@article {BKLRimp06,
    AUTHOR = {Boyallian, Carina and Kac, Victor G. and Liberati, Jose I. and
              Rudakov, Alexei},
     TITLE = {Representations of simple finite {L}ie conformal superalgebras
              of type {$W$} and {$S$}},
   JOURNAL = {J. Math. Phys.},
  FJOURNAL = {Journal of Mathematical Physics},
    VOLUME = {47},
      YEAR = {2006},
    NUMBER = {4},
     PAGES = {043513, 25},
      ISSN = {0022-2488,1089-7658},
   MRCLASS = {17B69 (17B65 81R10)},
  MRNUMBER = {2226350},
MRREVIEWER = {Pavel\ S.\ Kolesnikov},
       DOI = {10.1063/1.2191788},
       URL = {https://doi.org/10.1063/1.2191788},
}

@article {BKLjmp10,
    AUTHOR = {Boyallian, Carina and Kac, Victor G. and Liberati, Jos\'e{}
              I.},
     TITLE = {Irreducible modules over finite simple {L}ie conformal
              superalgebras of type {$K$}},
   JOURNAL = {J. Math. Phys.},
  FJOURNAL = {Journal of Mathematical Physics},
    VOLUME = {51},
      YEAR = {2010},
    NUMBER = {6},
     PAGES = {063507, 37},
      ISSN = {0022-2488,1089-7658},
   MRCLASS = {17B69 (17B10 81R10 81T40)},
  MRNUMBER = {2676484},
MRREVIEWER = {Markus\ Rosellen},
       DOI = {10.1063/1.3397419},
       URL = {https://doi.org/10.1063/1.3397419},
}

@article{Bai23,
title = {Transposed {P}oisson algebras, {N}ovikov-{P}oisson algebras and 3-{L}ie algebras},
journal = {Journal of Algebra},
volume = {632},
pages = {535-566},
year = {2023},
issn = {0021-8693},
doi = {https://doi.org/10.1016/j.jalgebra.2023.06.006},
author = {Chengming Bai and Ruipu Bai and Li Guo and Yong Wu}
}

@article {DAK98,
    AUTHOR = {D'Andrea, Alessandro and Kac, Victor G.},
     TITLE = {Structure theory of finite conformal algebras},
   JOURNAL = {Selecta Math. (N.S.)},
  FJOURNAL = {Selecta Mathematica. New Series},
    VOLUME = {4},
      YEAR = {1998},
    NUMBER = {3},
     PAGES = {377--418},
      ISSN = {1022-1824,1420-9020},
   MRCLASS = {17B68 (81R10 81T40)},
  MRNUMBER = {1654574},
MRREVIEWER = {Henrik\ Aratyn},
       DOI = {10.1007/s000290050036},
       URL = {https://doi.org/10.1007/s000290050036},
}
\bibliographystyle{plainurl}

\end{document}